\documentclass[11pt, a4paper reqno]{amsart}

\usepackage[dvipsnames]{xcolor}
\definecolor{darkblue}{rgb}{0,0,0.7}
\definecolor{darkred}{rgb}{0.7,0,0}
\usepackage[colorlinks=true, linkcolor=darkred, citecolor=darkblue, urlcolor=blue, pagebackref=true, breaklinks=true]{hyperref}
\usepackage[capitalise]{cleveref}

\usepackage{amscd}
\usepackage{amssymb,amsmath}

\usepackage{mathrsfs}
\usepackage{graphicx}
\usepackage{enumitem}
\usepackage{mathtools}
\usepackage[utf8]{inputenc}
\usepackage{listings}
\usepackage{lipsum,eso-pic,xcolor}
\usepackage{lineno}
\usepackage{tikz}
\usepackage{caption}
\usetikzlibrary{arrows,decorations.pathmorphing,backgrounds,positioning,fit}
\usetikzlibrary{positioning}

\usetikzlibrary{trees}
\usetikzlibrary{arrows}

\usepackage{placeins}
\usepackage{relsize}
\usepackage{todonotes}
\usepackage{soul}
\usepackage{enumitem}
\usepackage{setspace}
\usepackage{wrapfig}
\usepackage{etoolbox}
\BeforeBeginEnvironment{wrapfigure}{\setlength{\intextsep}{0pt}}
\newtheorem{proposition}{Proposition}[section]
\newtheorem{lemma}[proposition]{Lemma}
\newtheorem{theorem}[proposition]{Theorem}

\newtheorem{question}[proposition]{Question}

\newtheorem{conjecture}[proposition]{Conjecture}
\theoremstyle{definition}
\newtheorem{remark}[proposition]{Remark}

\newtheorem{example}[proposition]{Example}
\newtheorem{definition}[proposition]{Definition}
\newtheorem{notation}[proposition]{Notation}

\newcommand{\reg}{{\rm reg}}

\def\H{\mathcal{H}}

\def\K{\mathbb{K}}
\def\G{\mathcal{G}}

\def\P{\mathcal{P}}
\def\E{\mathcal{E}}
\def\N{\mathcal{N}}

\def\reg{\mathrm{reg}}
\def\lcm{\mathrm{lcm}}
\def\aim{\mathrm{aim}}

\def\l{\langle}
\def\r{\rangle}
\def\x{\mathbf x}

\def\Bl{\mathcal{B}_G}

\def\supp{\mathrm{supp}}

\begin{document}

\title[Admissible matchings and regularity of square-free powers]{Admissible matchings and the Castelnuovo-Mumford regularity of square-free powers}

\author{Trung Chau}
\address{Chennai Mathematical Institute, India}
\email{chauchitrung1996@gmail.com}

\author{Kanoy Kumar Das}
\address{Chennai Mathematical Institute, India}
\email{kanoydas@cmi.ac.in; kanoydas0296@gmail.com}

\author{Amit Roy}
\address{Chennai Mathematical Institute, India}
\email{amitiisermohali493@gmail.com}

\author{Kamalesh Saha}
\address{Chennai Mathematical Institute, India}
\email{ksaha@cmi.ac.in; kamalesh.saha44@gmail.com}


\vspace*{-0.4cm}
\begin{abstract}
    Let $I$ be any square-free monomial ideal, and $\mathcal{H}_I$ denote the hypergraph associated with $I$. Refining the concept of $k$-admissible matching of a graph defined by Erey and Hibi, we introduce the notion of generalized $k$-admissible matching for any hypergraph. Using this, we give a sharp lower bound on the (Castelnuovo-Mumford) regularity of $I^{[k]}$, where $I^{[k]}$ denotes the $k^{\text{th}}$ square-free power of $I$. In the special case when $I$ is equigenerated in degree $d$, this lower bound can be described using a combinatorial invariant $\mathrm{aim}(\mathcal{H}_I,k)$, called the $k$-admissible matching number of $\mathcal{H}_I$. Specifically, we prove that $\mathrm{reg}(I^{[k]})\ge (d-1)\mathrm{aim}(\mathcal{H}_I,k)+k$, whenever $I^{[k]}$ is non-zero. Even for the edge ideal $I(G)$ of a graph $G$, it turns out that $\mathrm{aim}(G,k)+k$ is the first general lower bound for the regularity of $I(G)^{[k]}$. In fact, when $G$ is a forest, $\mathrm{aim}(G,k)$ coincides with the $k$-admissible matching number introduced by Erey and Hibi. Next, we show that if $G$ is a block graph, then $\mathrm{reg}(I(G)^{[k]})= \mathrm{aim}(G,k)+k$, and this result can be seen as a generalization of the corresponding regularity formula for forests. Additionally, for a Cohen-Macaulay chordal graph $G$, we prove that $\mathrm{reg}(I(G)^{[2]})= \mathrm{aim}(G,2)+2$. Finally, we propose a conjecture on the regularity of square-free powers of edge ideals of chordal graphs.
\end{abstract}

\maketitle

\section{Introduction} 
    Matching theory of finite simple graphs has a rich history and numerous applications in different areas of mathematics and beyond (see, e.g., \cite{BM76,Edmonds}). Let $\H$ be a simple hypergraph with the vertex set $V(\H)$ and the edge set $E(\H)$. 
    A \emph{matching} in  $\H$ is a collection of disjoint edges in $\H$. The \emph{matching number} of $\H$, denoted by $\nu(\H)$, is the maximum size of a matching in $\H$. Very recently, the notion of matching in graph theory has been connected to the area of combinatorial commutative algebra via the introduction of square-free powers (see \cite{Bigdeli2018(1), CFL1, DRS20242, CMSimplicialForests, EHHS, EHHM12, FiHeHi, ficarraCM2024, KaNaQu, SaS1, Fakhari2024}, etc.) and this topic has been of great interests since \cite{EHHS}. The authors in \cite{Bigdeli2018(1)} introduced the notion of square-free powers for the class of edge ideals of simple graphs. In fact, this notion can be defined for any square-free monomial ideals or edge ideals of hypergraphs. Let $I\subseteq R=\K[x_1,\ldots , x_n]$ be a square-free monomial ideal, and $\G(I)$ be the set of minimal monomial generators of $I$. Then there is a hypergraph $\H$ such that $I=I(\H)$. The \emph{$k$-th square-free power} of $I$, denoted by $I^{[k]}$, is defined by 
    \begin{align*}
        I^{[k]}&=\l f\in I^k \mid f \text{ is square-free} \r\\
        &=\l u_1\cdots u_k\mid u_1,\ldots , u_k\in E(\H) \text{ is a matching in } \H \r
    \end{align*}

    The study of various algebraic invariants such as Castelnuovo-Mumford regularity (or simply, regularity), depth, projective dimension, etc., of powers of edge ideals of hypergraphs have been one of the fundamental areas of research in the last few decades. In this context, the notion of matching and induced matching in a hypergraph play a very crucial role. Loosely speaking, a matching in a hypergraph $\H$ corresponds to a regular sequence that is contained in the edge ideal $I(\H)$. On the other hand, the induced matching number of $\H$, i.e., the maximal cardinality of an induced matching, denoted by $\nu_1(\H)$, is an invariant that can be used to estimate lower bounds for the regularity of powers of edge ideals \cite{BCDMS2022}.
    In particular, the following general lower bound for the regularity of edge ideals of $d$-uniform hypergraphs is well-known: 
    \[\reg(I(\H)^k)\geq (d-1)(\nu_1(\H)-1)+kd.
    \]
    A relevant question in this direction is the following: Is it possible to find a square-free version of this general lower bound? To the best of our knowledge, there is no general lower bound for the regularity of square-free powers. 
    In \cite{EHHS}, the authors explored this question for the sub-class of edge ideals of graphs. Recall that a graph $G$ is a $2$-uniform simple hypergraph. The following analogous lower bound for the square-free powers of edge ideals of graphs is provided in \cite[Theorem~2.1]{EHHS}: 
    \[
    \reg(I(G)^{[k]})\geq \nu_1(G)+k \;\; \text{ for all }\; 1\leq k\leq \nu_1(G).
    \]
    It is important to note that this lower bound only works for $k\leq \nu_1(G)$, while it may happen that $I(G)^{[k]}\neq \l 0\r$ for $k>\nu_1(G)$ (for instance, when $G$ is complete). In this context, Erey and Hibi \cite{ErHi1} introduced a new combinatorial notion, called \emph{$k$-admissible matching} of a graph $G$, which can be seen as a generalization of induced matching.
    The \emph{$k$-admissible matching number} of $G$, denoted by $\aim(G,k)$, is the maximal size of a $k$-admissible matching in $G$. Interestingly, in \cite[Theorem 25]{ErHi1} the authors provided the following upper bound when $G$ is a forest:
    \[
    \reg(I(G)^{[k]})\leq \aim(G,k)+k \;\; \text{ for all }\; 1\leq k\leq \nu(G).
    \]
    They further showed that this upper bound is attained when $k=2$ and conjectured that this bound will be attained for any $k\leq \nu(G)$ in the case of forests. Later on, Crupi, Ficarra, and Lax \cite[Theorem~3.6]{CFL1} settled this conjecture.

    In this work, we refine and extend the notion of $k$-admissible matching of a simple graph $G$ to the realm of hypergraphs (see Definition~\ref{def: aim}). Using this new definition, we provide a general lower bound for square-free powers of any square-free monomial ideal. This extends and strengthens all known lower bounds for square-free powers in the literature. 

    \medskip
    \noindent
    \textbf{Theorem~\ref{thm: ordinary lower bound}.}
    Let $\H$ be a hypergraph and $k$ an integer where $1\leq k\leq \nu(\mathcal{H})$. Then for any generalized $k$-admissible matching $M$ of $\H$, we have
    \[
    \reg \left( I(\mathcal{H})^{[k]} \right) \geq |V(M)|-|M|+k.
    \]
    The above-mentioned lower bound is sharp, as seen in \Cref{prop:example-lower-bound-sharp}. In the case of $d$-uniform hypergraphs, we define a combinatorial invariant, called the \textit{$k$-admissible matching number} of $\H$, denoted by $\aim(\H,k)$ (see \Cref{def:aim-d-uniform}), which is analogous to the induced matching number in the context of square-free powers. In particular, $\aim(\H,1)$ coincides with $\nu_1(\H)$. As a consequence of our main theorem, we obtain the following lower bound for square-free powers of edge ideals of $d$-uniform hypergraphs and, in particular, simple graphs.

    \medskip
    \noindent
    \textbf{Theorem~\ref{thm: ordinary lower bound d uniform}.}
        Let $\H$ be a $d$-uniform hypergraph and $k$ an integer where $1\leq k\leq \nu(\mathcal{H})$. Then 
    \[
    \reg \left( I(\mathcal{H})^{[k]} \right) \geq (d-1)\aim(\H,k) +k.
    \]
    Specifically, for a simple graph $G$ and for any $1\leq k\leq \nu(G)$, we have 
    \[
    \reg \left( I(G)^{[k]}\right) \geq \aim(G,k)+k.
    \]

    \par
    We provide two classes of simple graphs where we actually get the equality in the formula of the lower bound of the regularity of square-free powers. Indeed, we show that for the class of block graphs, we have the following formula for the regularity: 
    
    \medskip
    \noindent
    \textbf{Theorem~\ref{block ordinary upper bound}.}
    Let $G$ be a block graph. Then for all $1\leq k\leq \nu(G)$, 
    \[
    \reg(I(G)^{[k]})=\aim(G,k)+k.
    \]
    Since the class of block graphs contains the class of forests, the above formula extends the regularity formula of square-free powers of edge ideals of forests provided in \cite{CFL1} and \cite{ErHi1}. Next, we consider the class of Cohen-Macaulay chordal graphs, which is another important class of chordal graphs. For edge ideals of this class, we prove a similar formula for the second square-free ~power:

    \medskip
    \noindent
    \textbf{Theorem~\ref{CM chordal regularity}.}
    Let $G$ be a Cohen-Macaulay chordal graph. Then 
    \[
    \reg(I(G)^{[2]})=\aim(G,2)+2.
    \]

    \par 
    Square-free powers are defined as the square-free part of ordinary powers. There is an analog for symbolic powers, first studied by Fakhari \cite{FakhariSymbSq-free}, which we shall call \emph{symbolic square-free powers}. We remark that an analog of our general lower bound can be obtained for symbolic square-free powers of any square-free monomial ideal \cite{CDRS-symbolic}.
    
    This article is organized as follows. In \Cref{sec: 2}, we recall the necessary definitions and results from combinatorics and commutative algebra. In \Cref{sec: 3}, we prove the general lower bound for square-free powers. In \Cref{sec: 4} and \ref{sec: 5}, we show that the general lower bound is attained for all the square-free powers of edge ideals of block graphs and the second square-free power of the edge ideal of Cohen-Macaulay chordal graphs, respectively. Finally, \Cref{section 6} explores an implication of \Cref{thm: ordinary lower bound} in the setting of recently introduced matching powers \cite{ErFi1} of (not necessarily square-free) monomial ideals. This section also presents a few open problems related to the regularity of square-free powers of edge ideals of various classes of graphs. In particular, building on our results and some computational evidence, we propose the following conjecture:

    \medskip
    \noindent
    \textbf{Conjecture~\ref{conj: chordal reg}.}
        Let $G$ be a chordal graph. Then, for all $1\leq k\leq \nu(G)$, we have
        \[ \reg(I(G)^{[k]})=\aim(G,k)+k. \]

\section{Preliminaries}\label{sec: 2}

In this section, we quickly recall some preliminary notions related to combinatorics and commutative algebra. For any undefined term, the reader is advised to look at \cite{BM76,Edmonds,RHV}.

\subsection{Combinatorics:}
    A (simple) hypergraph $\H$ is a pair $\H=(V(\H),E(\H))$, where $V(\H)$ is a finite set, called the \textit{vertex set} of $\H$, and $E(\H)$ is a collection of subsets of $V(\H)$ such that no two elements contain each other, called the \textit{edge set} of $\H$. The elements of $V(\H)$ are called the vertices of $\H$, and the elements of $E(\H)$ are called the edges of $\H$. An {\it induced sub-hypergraph} of $\H$ on the vertex set $W\subseteq V(\H)$, denoted by $\H[W]$, is a hypergraph with $V(\H[W])=W$ and $E(\H[W])=\{\E\in E(\H)\mid \E \subseteq W\}$. For $W\subseteq V(\H)$, the hypergraph $\H[V(\H)\setminus W]$ is also denoted by $\H\setminus W$.  
If $W=\{x\}$ for some $x\in V(\H)$, then $\H\setminus \{x\}$ is simply denoted by $\H\setminus x$. A \textit{matching} of $\H$ is a collection of edges $M$ of $\H$ such that $\E\cap \E' =\emptyset$ for any two distinct edges $\E,\E'\in M$.  A matching $M$ of $\H$ is called an {\it induced matching} of $\H$ if $E(\H[\cup_{\E\in M}\E])=M$. The {\it induced matching number} of $\H$, denoted by $\nu_1(\H)$, is the maximum cardinality of an induced matching in $\H$. \par

A (simple) \textit{graph} $G$ is nothing but a hypergraph whose edges are of cardinality two. For $A\subseteq V(G)$, the set of {\it closed neighbors} of $A$ is defined as $N_{G}[A]\coloneqq A\cup \{x\in V(G)\mid \{x,y\}\in E(G) \text{ for some } y\in A\}$. By a set of \textit{open neighbors} of $A$, we mean $N_{G}(A)\coloneqq N_{G}[A]\setminus A$. If $A=\{x\}$, then $N_G(\{x\})$ is simply denoted by $N_G(x)$. For $x\in V(G)$, the number $|N_G(x)|$ is called the {\it degree} of $x$, and is denoted by $\deg(x)$. Given two vertices $x,y\in V(G)$, an {\it induced path of length $n$} between $x$ and $y$ is a collection of vertices $P:a_1,\ldots,a_n$ such that $a_1=x,a_n=y$, and $\{a_i,a_j\}\in E(G)$ if and only if $j=i-1$ or $j=i+1$. A \textit{complete graph} on $n$ vertices, denoted by $K_n$, is a graph such that there is an edge between every pair of vertices. A \textit{clique} of a graph $G$ is a complete induced subgraph of $G$. A vertex $x\in V(G)$ is said to be a \textit{free vertex} of $G$ if $G[N_{G}(x)]$ is a complete graph. We say two edges $e$ and $e'$ of $G$ form a \textit{gap} in $G$ if $\{e,e'\}$ forms an induced matching in $G$. The complement of a graph $G$, denoted by $G^c$, is a graph such that $V(G^c)=V(G)$ and $E(G^c)=\{\{x,y\}\mid \{x,y\}\not\in E(G)\}$.

\par
A {\it cycle} of length $m$, denoted by $C_m$, is a graph such that $V(C_m)=\{x_1,\ldots,x_m\}$ and $E(C_m)=\{\{x_1,x_m\},\{x_i,x_{i+1}\}\mid 1\le i\le m-1\}$. A graph $G$ is called a \textit{forest} if it does not contain any induced cycle. A connected forest is known as a \textit{tree}.  A graph is called \textit{chordal} if it does not contain an induced $C_n$ for $n\geq 4$. A graph is said to be \textit{weakly chordal} if the graph and its complement both do not have any induced cycle $C_n$ for $n\geq 5$. Note that any chordal graph is also weakly chordal.

\subsection{Commutative Algebra} For a graded ideal $I\subseteq R$, the graded minimal free resolution of $I$ is an exact sequence
\[
\mathcal F_{\cdot}: \,\, 0\rightarrow F_p\xrightarrow{\partial_{p}}\cdots\xrightarrow{\partial_{2}} F_1\xrightarrow{\partial_1} F_0\xrightarrow{\partial_0} I\rightarrow 0, 
\]
where $F_i=\oplus_{j\in\mathbb N}R(-j)^{\beta_{i,j}(I)}$ is a free $R$-module for each $i\ge 0$ and $R(-j)$ denotes the polynomial ring $R$ with grading shifted by $j$. The numbers $\beta_{i,j}(I)$ are called the $i^{th}$ $\mathbb N$-graded {\it Betti numbers} of $I$ in degree $j$. The {\it Castelnuovo-Mumford regularity} (in short, regularity) of $I$, denoted by $\reg(I)$, is defined as $\reg(I)\coloneqq \max\{j-i\mid \beta_{i,j}(I)\neq 0\}$. If $I$ is the zero ideal, then we adopt the convention that $\reg(I)=0$. 
The following two well-known results regarding the regularity of a graded ideal will be used frequently in the subsequent sections.

\begin{lemma}[{\cite[Lemma 2.10]{DHS}}]\label{regularity lemma}
    Let $I \subseteq R$ be a monomial ideal, $m$ be a monomial of degree $d$ in $R$, and $x$ be an indeterminate in $R$. Then 
\begin{enumerate}[label=(\roman*)]
    \item $\reg(I+\l x\r)\le \reg(I) $,

    \item $\reg(I:x)\le \reg(I)$,

    \item $\reg I \le \max\{\reg(I : m) + d, \reg( I+ \l m\r)\}$.
\end{enumerate}
    \end{lemma}

\begin{lemma}[{\cite{RHV}}]\label{reg sum}
        Let $I_1\subseteq R_1=\mathbb K[x_1,\ldots,x_m]$ and $I_2\subseteq R_2=\mathbb K[x_{m+1},\ldots, x_n]$ be two graded ideals. Consider the ideal $I=I_1R+I_2R\subseteq R=\mathbb K[x_1,\ldots,x_n]$. Then 
        \[ \reg(I)=\reg(I_1)+\reg(I_2)-1.
        \]
    \end{lemma}

Let $\H$ be a hypergraph on the vertex set $V(\H)=\{x_1,\ldots,x_n\}$. Identifying the vertices of $\H$ as indeterminates, we consider the polynomial ring $R=K[x_1,\ldots,x_n]$ over a field $K$. Corresponding to any $F\subseteq V(\H)$, we denote the square-free monomial $\prod_{x_i\in F}x_i$ of $R$ by $\x_{F}$. The edge ideal of $\H$, denoted by $I(\H)$, is a square-free monomial ideal of $R$ defined as 
\[ I(\H)\coloneqq \l \x_{\E}\mid \E\in E(\H)\r. \]
Note that any square-free monomial ideal can be seen as the edge ideal of a hypergraph. 
\par 

\newpage
\section{$k$-admissible matching and a general lower bound}\label{sec: 3}

The induced matching number of a hypergraph $\H$ is a combinatorial invariant that has been used to estimate lower bounds for the regularity of ordinary and symbolic powers of the edge ideal $I(\H)$ \cite{BCDMS2022}. In the case of square-free powers, a similar attempt has been made in \cite[Theorem 2.1]{EHHS}, where the authors have shown that $\reg(I(G)^{[k]})\geq \nu_1(G)+k$ for all $1\leq k\leq \nu_1(G)$. Later on, as a generalization of induced matching, $k$-admissible matching was introduced by Erey and Hibi. Then it was proved that for the class of forests, the equality $\reg(I(G)^{[k]})=\aim(G,k)+k$ holds for all $1\leq k\leq \nu(G)$ (see \cite[Theorem 3.6]{CFL1} and \cite[Theorem 25]{ErHi1}).  Furthermore, it follows from the work of Fakhari \cite[Theorem 4.3]{Fakhari2024}, that for a Cameron-Walker graph $G$, $\reg(I(G)^{[k]})=\aim(G,k)+k$ for all $1\leq k\leq \nu(G)$. Based on these results, one can expect that for an arbitrary graph $G$, the regularity of $I(G)^{[k]}$ is bounded below by the combinatorial invariant $\aim(G,k)+k$. 

Inspired by this idea, in this section, we refine and extend their definition of $k$-admissible matching for simple graphs to all hypergraphs. To avoid any ambiguity, we call this refined notion of matching to be generalized $k$-admissible matching of a hypergraph and use the same notations. We use $\sqcup$ to denote disjoint unions. For a matching $M$, let $V(M)$ denote the set of vertices that belong to the hyperedges in $M$.

\begin{definition}\label{def: aim}
    Let $\H$ be a hypergraph and $M$ a matching of $\H$. Then for any $1\leq k\leq \nu(\H)$, we say that $M$ is a \textit{generalized $k$-admissible matching} if there is a partition $M=M_1\sqcup \cdots \sqcup M_r$ such that 
    \begin{enumerate}
        \item for any edge $\E\in \H[V(M)]$, we have $\E\in E(\H[V(M_i)])$ for some $i\in [r]$;
        \item $k\leq |M|=\sum_{i=1}^r |M_i|\leq r+k-1$;
        \item for any $i\in [r]$ and any matching $M_i'$ in $\H[V(M_i)]$ with $|M_i'|=|M_i|$, one has $V(M_i')=V(M_i)$.
    \end{enumerate}     
\end{definition}

\noindent\textbf{Existence of generalized $k$-admissible matching:} 
Let $\H$ be a hypergraph and $1\leq k\leq \nu(\H)$. Take a matching $M'$ of $\H$ with $\vert M'\vert =k$. Now, consider the induced sub-hypergraph $\H'=\H[V(M')]$ and pick a minimal generator $u$ of the ideal $I(\H')^{[k]}$. Then the induced sub-hypergraph $\H[\supp(u)]$ has a matching number $k$, and any matching of size $k$ in $\H[\supp(u)]$ consists of all the vertices in $\supp(u)$. Thus, any matching $M$ with $|M|=k$ in $\H[\supp(u)]$ will give a generalized $k$-admissible matching of $\H$ by taking the partition $M=M_1$.

\medskip

Our aim in this section is to give a lower bound for $\reg I(\H)^{[k]}$ for any $1\le k\le\nu(\H)$ using generalized $k$-admissible matching. We present some basic observations for this new concept.

\begin{remark}\label{aim property all hypergraphs} Let $\H$ be a hypergraph and $M\subseteq E(\H)$ is a matching of $\H$. 
     \begin{enumerate}
     \item A matching $M$ of $\H$ is a generalized $1$-admissible matching if and only if $M$ is an induced matching of $\H$.
    \item  The condition $(3)$ in \Cref{def: aim} implies that $\nu(\H[V(M_i)])=|M_i|$ for any $i\in[r]$. Indeed, suppose otherwise that there exists a matching $M'$ in $\H[V(M_i)]$ such that $|M'|>|M_i|$. Then consider a proper subset $M''$ of $M'$ such that $|M''|=|M_i|$. This contradicts \cref{def: aim} (3), as desired.
    \item The condition $(3)$ in \Cref{def: aim} is equivalent to the following fact: for any $i\in [r]$, $I(\H[V(M_i)])^{[|M_i|]}$ is generated by one monomial $\prod_{x\in V(M_i)} x$. This is immediate from the previous remark and \cref{def: aim} (3).
\end{enumerate}
\end{remark}

The following result follows from verifying the conditions in \cref{def: aim}; thus, we state it without proof.

\begin{proposition}\label{admissible result2}
    Let $\mathcal{H}$ be a hypergraph and $\mathcal{H}_1$ an induced sub-hypergraph of $\mathcal{H}$. 
    Then, any generalized $k$-admissible matching in $\H_1$ is a generalized $k$-admissible matching of $\H$.
\end{proposition}

Our primary tool to establish the regularity lower bound is the idea of Betti splitting introduced in \cite{FranciscoHaVanTuyl2009}.

\begin{definition}[{\cite[Definition 1.1]{FranciscoHaVanTuyl2009}}]\label{def:betti-splittings}
    Let $I,J, \text{ and } K$ be monomial ideals such that $\G(I)$ is the disjoint union of $\G(J)\text{ and } \G(K)$. Then $I=J+K$ is a \textit{Betti splitting} if 
    \begin{align}\label{splitting formula}
    \beta_{i,j}(I)=\beta_{i,j}(J)+\beta_{i,j}(K)+\beta_{i-1,j}(J\cap K) \;\;\;\text{ for all } i,j.
    \end{align}
\end{definition}

Although the notion of Betti splitting was formally introduced in \cite{FranciscoHaVanTuyl2009}, it first appeared in the work of Eliahou and Kervaire \cite{EliKer}. In fact, in this paper, we use the concept of splittable monomial ideal as defined in \cite{EliKer}.

\begin{definition}[{cf. \cite{EliKer}}]\label{def:splittings}
        A monomial ideal $I$ is said to be \emph{splittable} if $I=J+K$, where $J$ and $K$ are two non-zero ideals and $\G (I)=\G(J)\sqcup \G(K)$, such that there exists a splitting function 
    \begin{align*}
    (\phi,\psi): \G(J\cap K)&\rightarrow \G(J)\times \G(K)\\
    w &\mapsto (\phi(w),\psi(w))
    \end{align*}
    which satisfies the following two properties:
    \begin{enumerate}[label=(\alph*)]
        \item for each $w\in \G(J\cap K)$, $w=\mathrm{lcm} (\phi(w),\psi(w))$;
        \item for each $S\subseteq \G(J\cap K)$, both $\mathrm{lcm}(\phi(S))$ and $\mathrm{lcm}(\psi(S))$ strictly divide $\mathrm{lcm}(S)$.
    \end{enumerate}
\end{definition}
\begin{remark}\label{EL splitting remark}
  The fact that for a splittable monomial ideal $I=J+K$ in \Cref{def:splittings}, \Cref{splitting formula} holds, follows from \cite[Proposition 3.2]{Fatabbi2001}.  
\end{remark}

It is evident from \Cref{splitting formula} that Betti splittings allow us to study the graded Betti numbers of a splittable monomial ideal $I$ in terms of the graded Betti numbers of the smaller ideals $J$ and $K$. We provide an important application of this below.

\begin{lemma}\label{lem:Betti-splittings-sequences}
    Let $\{I_i\}_{i\in \mathbb{N}}$ be a sequence of proper square-free monomial ideals in $\K[x_1,\dots, x_{N_1}]$ such that for any $i\in \mathbb{N}$ and any $f\in I_i$, we have $f/x\in I_{i-1}$ for any $x\in \supp(f)$. Let $\{L_i\}_{i\in \mathbb{N}}$ be a sequence of square-free monomial ideals in $\K[y_1,\dots, y_{N_2}]$ with the same condition. Let $n,k,l$ be three integers where $l<k\leq n$. Set
    \begin{align*}
        J&\coloneqq I_k L_{n-k},\\
        K&\coloneqq I_{k-1} L_{n-k+1} + \cdots + I_lL_{n-l}.
    \end{align*}
    Assume that $J\neq 0$. Then $J+K$ is a Betti splitting. In particular, we have
    \[
    \reg\left( J+K \right) \geq  \max \{ \reg \left( I_k \right) + \reg\left(  L_{n-k} \right),\
        \reg \left( I_k\right) + \reg \left( L_{n-k+1} \right) -1   \}.
    \]
\end{lemma}

\begin{proof}
    First we show that $\G(J)\cap \G(K)=\emptyset$. For this, it suffices to show that whenever $I_i$ is non-zero for some $i\in\mathbb N$, then $I_i\subsetneq I_{i-1}$, as an analogous statement for $L_i\neq \l 0\r$ would follow. Indeed, let $f\in \G(I_i)$, then $f/x\in I_{i-1}$ by our hypotheses for any $x\in \supp(f)$. Since $f\in \G(I)$, we have $f/x\notin I_{i}$. Thus, we have $\G(J)\cap \G(K)=\emptyset$.

    Note that, if $K=\l 0\r$, then $J=J+K$ is trivially a Betti splitting, and we have the result. Thus, we may assume $K\neq\l 0\r$. We next show that $J\cap K= I_k L_{n-k+1}$. Indeed, the reverse inclusion is clear. For the inclusion, since $I_i$ and $L_j$ share no variable for any $i$ and $j$, any minimal monomial generator of $J\cap K$ is of the form $\lcm(fg,f'g')$ where $f\in I_k, g\in L_{n-k}, f'\in I_{r},$ and $g'\in L_{n-r}$ for some $l \leq r<k$. It is then clear that $\lcm(fg,f'g')\in I_kL_{n-k+1}$, and thus the inclusion follows.

    Let $(>)$ be a fixed monomial ordering on $\{x_1,\dots, x_{N_1}, y_1,\dots y_{N_2}\}$ and $(>_{lex})$ be the lex ordering on the monomials induced from $(>)$.

    Consider any generator $f\in \G(J\cap K) = \G(I_kL_{n-k+1})$. Set $f=f_1f_2$ where $f_1\in \G(I_k)$ and $f_2\in \G(L_{n-k+1})$. Set 
    \begin{align*}
        x_f&\coloneqq \max_{>} \supp(f_1),\\
        y_f&\coloneqq \max_{>} \supp(f_2),\\
        f_I &\coloneqq \max_{>_{lex}} \{ g \colon x_f \mid g \text{ and } g\mid f_1,  \text{ and } f_1/g\in \G(I_{k-1}) \},\\
        f_{L} &\coloneqq \begin{multlined}[t]
            \max_{>_{lex}} \{g  \colon y_f \mid g \text{ and } g\mid f_2, \text{ and } f_2/g\in \G(L_{n-k} )  \}.
        \end{multlined}
    \end{align*}
    The existence of $f_I$ and $f_L$ is due to the hypotheses.
   
    Let $(\phi,\psi)$ define a map as follows: 
    \begin{align*}
        (\phi,\psi) \colon \G(J\cap K) &\to \G(J) \times \G(K)\\
        f&\mapsto \left( \frac{f}{f_{L}} , \frac{f}{f_I} \right).
    \end{align*}
    Due to our constructions, the map is well-defined. We show that $(\phi,\psi)$ indeed defines a splitting ~function:
    \begin{enumerate}
        \item[(a)] for each $f\in \G(J\cap K)$, we have $f=\lcm(\frac{f}{f_L},\frac{f}{f_{I}}) = \lcm(\phi(f),\psi(f))$ since $f_L$ and $f_{I}$ have disjoint~support;
        \item[(b)]  for each $S =\{h_1,\dots, h_s\}   \subseteq \G(J\cap K)$, both $\lcm(\phi(S)) $ and $\lcm(\psi(S))$ strictly divide $\lcm(S)$ since by design, we have 
        \begin{align*}
            x_S,y_S\mid \lcm(h_1,\dots, h_s) = \lcm(S),\tag*{(1)}\\
            x_S\nmid   \lcm(\frac{h_1}{{h_1}_I}, \dots, \frac{h_s}{{h_s}_I}) = \lcm(\phi(S)),\tag*{(2)}\\
            y_S\nmid   \lcm(\frac{h_1}{{h_1}_{L}}, \dots, \frac{h_s}{{h_s}_{L}}) = \lcm(\psi(S)), \tag*{(3)}
        \end{align*}
        where 
        \begin{align*}
            x_S\coloneqq \max_{>} \{ x_{{h_i}} \colon i\in [s] \},\\
            y_S\coloneqq \max_{>} \{ y_{{h_i}} \colon i\in [s] \}.
        \end{align*}
        Indeed, (1) follows from the definition of $x_S$ and $y_S$. Regarding (2), observe that if $x_S\in \supp(h_i)$ for some $i\in [s]$, then we have $x_S=x_{h_i}\mid {h_i}_I$. This implies that $x_S\nmid \frac{h_i}{{h_i}_I}$ for any $i\in [s]$. Thus, (2) follows. The proof for (3) is similar by symmetry. 
    \end{enumerate} 
    Thus, $I=J+K$ is a Betti splitting. The last statement follows from \Cref{splitting formula}.
    \end{proof}

We now apply \Cref{lem:Betti-splittings-sequences} in the special case of square-free powers of edge ideals to obtain an important regularity formula needed for the general lower bound of square-free powers. Here we adopt the convention that $I^{[l]}=R$ for any $l\le 0$.

\begin{lemma}\label{lem:square-free-power-splittings}
    Let $\mathcal{H}$ and $\mathcal{H}'$ be hypergraphs that have at least one edge each and no common vertex, and $k$ a positive integer such that $\nu(\mathcal{H})+1\leq k\leq \nu(\mathcal{H})+ \nu(\mathcal{H}')$. Then
    \[
    I(\mathcal{H}\sqcup \mathcal{H}')^{[k]} = J + K, 
    \]
    where 
    \begin{align*}
        J&= I(\mathcal{H})^{[\nu(\mathcal{H})]} I(\mathcal{H}')^{[k-\nu(\mathcal{H})]} , \\
        K&= I(\mathcal{H})^{[\nu(\mathcal{H})-1]} I(\mathcal{H}')^{[k-\nu(\mathcal{H})+1]} + \cdots + I(\mathcal{H})^{[k-\nu(\mathcal{H}')]} I(\mathcal{H}')^{[\nu(\mathcal{H}')]},
    \end{align*}
    is a Betti splitting. In particular, we have
    \begin{multline*}
         \reg\left( I(\mathcal{H}\sqcup \mathcal{H}')^{[k]} \right) \geq  \max \{ \reg \left( I(\mathcal{H})^{[\nu(\mathcal{H})]} \right) + \reg\left( I(\mathcal{H}')^{[k-\nu(\mathcal{H})]}\right),\\
        \reg \left( I(\mathcal{H})^{[\nu(\mathcal{H})]} \right) + \reg \left( I(\mathcal{H}')^{[k-\nu(\mathcal{H})+1]} \right) -1   \}.
    \end{multline*}
\end{lemma}

\begin{proof}
    The condition $\nu(\mathcal{H})+1\leq k\leq \nu(\mathcal{H})+ \nu(\mathcal{H}')$ implies that $J\neq \l 0\r$. The claim that $I(\mathcal{H}\sqcup \mathcal{H}')^{[k]}=J+K$ follows immediately from definition. It is clear that for any monomial $f\in I(\mathcal{H})^{[k]}$, we have $f/x\in I(\mathcal{H})^{[k-1]}$ for any $x\in \supp(f)$. The results then follow from Lemma~\ref{lem:Betti-splittings-sequences}.
\end{proof}

We are now ready to prove the main result of this section. Before that, we recall the result \cite[Corollary 1.3]{EHHS} on the regularity in the context of taking induced subgraphs. Note that while the authors stated the result for only simple graphs, their proof also works for any hypergraph. Thus, we have the following:

\begin{lemma}\label{lem:restriction-lemma}
    Let $\mathcal{H}$ be a hypergraph, $\mathcal{H}_1$ an induced sub-hypergraph of $\H$, and $k$ a positive integer. Then
    \begin{align*}
        \reg\left( I(\mathcal{H}_1)^{[k]} \right) \leq \reg\left( I(\mathcal{H})^{[k]} \right).
    \end{align*}
\end{lemma}

We first consider the special case when hyperedges of $\H$ are pairwise disjoint.

\begin{proposition}\label{prop:example-lower-bound-sharp}
    Let $\H$ be a hypergraph such that any two hyperedges of $\H$ have an empty intersection. Then for any integer $k$ where $1\leq k\leq \nu(\mathcal{H})$, we have
    \begin{align*}
        \reg \left( I(\mathcal{H})^{[k]} \right) &= \max\{ |V(M)|-|M|\colon \text{$M$ is a generalized $k$-admissible matching of $\H$}\}+k\\
        &= |V(\H)| - |E(\H)|+k.
    \end{align*}
\end{proposition}

\begin{proof}
    Set $E(\H)=\{\E_1,\dots, \E_q\}$, and for any $i\in [q]$, set $m_i=\prod_{x\in \E_i} x$. As any two hyperedges of $\H$ have empty intersection, the monomials $\{m_i\}_{i=1,\dots, q}$ have pairwise disjoint supports. In particular, $I(\H)=(m_1,\dots, m_q)$ is a complete intersection. It is straightforward that that $E(\H)$ itself is a generalized $k$-admissible matching for any $1\leq k\leq \nu(\H)$. Therefore, we have
    \[
    \max\{ |V(M)|-|M|\colon \text{$M$ is a $k$-admissible matching of $\H$}\} = |V(E(\H))| - q = \sum_{i=1}^q |\E_i| - q.
    \]
    It now suffices to show that
    \[
    \reg \left( I(\mathcal{H})^{[k]} \right) = \sum_{i=1}^q |\E_i| - q+k,
    \]
    by induction on $q$ and $k$. If $q=1$, then we must have $k=1$. The result then follows immediately. On the other hand, if $k=1$, then as $I(\H)^{[k]}$ is a complete intersection, it is known that
    \[
    \reg \left( I(\mathcal{H})^{[k]} \right)  = \reg \left( I(\mathcal{H}) \right) =  \sum_{i=1}^q \deg m_i - (q-1) = \sum_{i=1}^q |\E_i| - q+k,
    \]
    as desired. We can now assume that  $q,k\geq 2$.
    Set $J\coloneqq m_q(m_1,\dots, m_{q-1})^{[k-1]}$ and $K\coloneqq (m_1,\dots, m_{q-1})^{[k]}$. It is straightforward that $J\cap K = m_q (m_1,\dots, m_{q-1})^{[k]}$. By $\cref{lem:square-free-power-splittings}$, $I(\H)^{[k]}=J+K$ is a Betti splitting, and thus by \cref{def:betti-splittings}, we have 
    \begin{multline*}
        \reg \left( I(\mathcal{H})^{[k]} \right) = \max\{ \reg \left( m_q(m_1,\dots, m_{q-1})^{[k-1]} \right) ,\quad \reg \left( (m_1,\dots, m_{q-1})^{[k]} \right) ,\\
        \reg \left( m_q(m_1,\dots, m_{q-1})^{[k]} \right) -1 \}.
    \end{multline*}
    By the induction hypotheses, we have 
    \begin{align*}
        \reg \left( (m_1,\dots, m_{q-1})^{[k-1]} \right) &= \sum_{i=1}^{q-1} |\E_i| - (q-1)+(k-1) = \sum_{i=1}^{q-1} |\E_i| - q+k,\\
        \reg \left( (m_1,\dots, m_{q-1})^{[k]} \right) &= \sum_{i=1}^{q-1} |\E_i| - (q-1)+k = \sum_{i=1}^{q-1} |\E_i| - q+k+1.
    \end{align*}
    Therefore, we have
    \begin{align*}
        \reg \left( I(\mathcal{H})^{[k]} \right) &=\begin{multlined}[t]
         \max\{\reg\left( m_q\right)+ \reg \left((m_1,\dots, m_{q-1})^{[k-1]} \right) ,\quad \reg \left( (m_1,\dots, m_{q-1})^{[k]} \right) ,\\
        \reg\left( m_q\right) +\reg \left((m_1,\dots, m_{q-1})^{[k]} \right) -1 \}
        \end{multlined}\\
        &=\max\{ \sum_{i=1}^{q} |\E_i| - q+k ,\quad  \sum_{i=1}^{q-1} |\E_i| - q+k+1,\quad
         \sum_{i=1}^{q} |\E_i| - q+k \}\\
         &= \sum_{i=1}^{q} |\E_i| - q+k,
    \end{align*}
    as desired. This concludes the proof.
\end{proof}

We are now ready to prove our main theorem of this section.

\begin{theorem}\label{thm: ordinary lower bound}
    Let $\H$ be a hypergraph and $k$ an integer where $1\leq k\leq \nu(\mathcal{H})$. Then for any generalized $k$-admissible matching $M$ of $\H$, we have
    \[
    \reg \left( I(\mathcal{H})^{[k]} \right) \geq |V(M)|-|M|+k.
    \]
    Consequently, we have
    \[
    \reg \left( I(\mathcal{H})^{[k]} \right) \geq \max\{ |V(M)|-|M|\colon \text{$M$ is a generalized $k$-admissible matching of $\H$}\}+k.
    \]
\end{theorem}
\begin{proof}
    Without loss of generality, assume that $\mathcal{H}$ has no isolated vertex. We induce on $k$ and $|E(\mathcal{H})|$. If $k=1$, then the result follows from \cite[Corollary 3.9]{MoreyVillarreal}. If $\mathcal{H}$ has one edge, then we must have $k=1$, and the result follows trivially. By induction, we can assume that $k\geq 2$ and $\mathcal{H}$ has at least two edges, and whenever $k'<k$ or $|E(\mathcal{H}')|<|E(\mathcal{H})|$, we have
    \[
    \reg \left( I(\mathcal{H}')^{[k']} \right) \geq |V(M')|-|M'|+k'
    \]
    where $M'$ is any generalized $k'$-admissible matching of $\H'$. Now assume that $M=\sqcup_{i=1}^r M_i$ is a generalized $k$-admissible matching  of $\H$ where the conditions in \cref{def: aim} hold. Set $a=|V(M)| - |M|$, and $a_i=|M_i|$ for any $i\in [r]$. By \cref{lem:restriction-lemma}, it suffices to show that
    \begin{equation}
        \reg\left( I(\mathcal{H}[V(M)])^{[k]} \right) \geq a+k.
    \end{equation}
    Suppose that $r=1$. Then $k\leq |M|=a_1\leq k$, and thus $|M|=a_1=k$. Since the ideal $I(\mathcal{H}[V(M)])^{[k]}$ is the principal ideal generated by the product of all variables in $M$, we have
    \[
    \reg\left( I(\mathcal{H}[V(M)])^{[k]} \right) = |V(M)|= \left(|V(M)| - |M|\right)+ |M| = a+k,
    \]
    as desired. Now we can assume that $r\geq 2$.  Due to condition (1) of \cref{def: aim}, the graph $\mathcal{H}[V(M)]$ is exactly the disjoint union of the hypergraphs $\H[V(M_i)]$ where $i$ ranges in $[r]$. Set $\H_1\coloneqq \mathcal{H}[V(M_1)]$, $M'=\sqcup_{i=2}^r M_i$, and $\H_2\coloneqq \mathcal{H}[V(M')] = \sqcup_{i=2}^r \mathcal{H}[V(M_i)]$. Then $\H_1$ and $\H_2$ share no vertex. If $a_i=1$ for any $i\in [r]$, then $M$ is a matching whose hyperedges are pairwise disjoint. The result then follows from \cref{prop:example-lower-bound-sharp}. We can now assume that $a_r\geq 2$. We then have $k-\nu(H_1)=k-a_1\geq \sum_{i=2}^r a_i - (r-1)> 0$ and $k\leq a_1+ \sum_{i=2}^ra_i=\nu(H_1)+\nu(H_2)$ by conditions (2) and (3) of \cref{def: aim}. Thus by \cref{lem:square-free-power-splittings},  we have
    \begin{equation}\label{equ:regularity-Betti-regular-power}
        \reg \left( I(\mathcal{H}[V(M)]) \right)^{[k]} \geq \max \{ |V(M_1)| + \reg \left( I(\H_2)^{[k-a_1]} \right), \quad |V(M_1)| + \reg \left( I(\H_2)^{[k-a_1+1]} \right) -1  \},
    \end{equation}
    where we used the fact that $\reg \left( I(\H_1)^{[a_1]}\right)$ is the principal ideal generated by a monomial of degree $|V(M_1)|$, from condition (3) of \cref{def: aim}. 
    Recall that we have $|M|=\sum_{i=1}^r a_i \leq r+k-1$. If $|M|<r+k-1$, then
    \[
    |M'| = \sum_{i=2}^r a_i = |M|-a_1 < (r+k-1)-a_1 = (r-1)+ (k-a_1),
    \]
    and hence
    \[
     |M'| \leq (r-1) + (k-a_1) - 1.
    \]
    On the other hand, since $k\leq |M|=a_1+|M'|$, we have $k-a_1\leq |M'|$. Thus, $M'=\sqcup_{i=2}^r M_i$ is a generalized $(k-a_1)$-admissible matching in $\H_2$. By induction, we have
    \begin{align*}
        \reg\left( I(\H_2)^{[k-a_1]} \right) \geq (|V(M')| - |M'|) + (k-a_1) &= (a-|V(M_1)| + a_1) + (k-a_1) \\
        &= a -|V(M_1)|+ k.
    \end{align*}
    Thus, by (\ref{equ:regularity-Betti-regular-power}), we have
    \[
    \reg \left( I(\mathcal{H}[V(M)])^{[k]} \right) \geq |V(M_1)| + \reg\left( I(\H_2)^{[k-a_1]} \right) \geq  |V(M_1)| + \left(a -|V(M_1)|+ k\right)= a+k, 
    \]
    as desired. Now we can assume that $|M|=r+k-1$. We have
    \[
    |M'| = \sum_{i=2}^r a_i = |M|-a_1 = (r+k-1)-a_1 = (r-1)+ (k-a_1+1)-1.
    \]
    Since $r\geq 2$, we have
    \[
    k-a_1+1\leq  (r-2) + (k-a_1+1) =  |M'| = (r-1)+ (k-a_1+1)-1.
    \]
    Thus, $M'=\sqcup_{i=2}^r M_i$ is a generalized $(k-a_1+1)$-admissible matching in $\H_2$. By induction, we ~have
    \begin{align*}
        \reg\left( I(\H_2)^{[k-a_1+1]} \right) \geq (|V(M')| -|M'|) + (k-a_1+1) &= (a-|V(M_1)| + a_1) + (k-a_1+1) \\
        &= a-|V(M_1)|+k+1.
    \end{align*}
    Thus, by (\ref{equ:regularity-Betti-regular-power}), we have
    \begin{align*}
        \reg \left( I(\mathcal{H}[V(M)])^{[k]} \right) &\geq |V(M_1)| + \reg\left( I(\H_2)^{[k-a_1+1]} \right) -1\\
        &\geq  |V(M_1)| + \left(a-|V(M_1)|+k+1 \right)-1 \\
        &= a+k, 
    \end{align*}
    as desired. This concludes the proof.
\end{proof}

In particular, the lower bound in this result is sharp, as equality is achieved in the case where $I(\H)$ is a complete intersection (\cref{prop:example-lower-bound-sharp}).

This lower bound can be related to more familiar concepts in the literature in the case of $d$-uniform hypergraphs and, in particular, simple graphs.  

\begin{definition}\label{def:aim-d-uniform}
    Let $\H$ be a $d$-uniform hypergraph. For any $1\leq k\leq \nu(\H)$, we say a matching $M$ of $\H$ is a \textit{$k$-admissible matching} if there is a partition $M=M_1\sqcup \cdots \sqcup M_r$ such that 
    \begin{enumerate}
        \item for any edge $\E \text{ in } \H[V(M)]$, we have $\E\in E(\H[V(M_i)])$ for some $i\in [r]$;
        \item $|M|\leq r+k-1$.
    \end{enumerate}
    Now, we define the \textit{$k$-admissible matching number} of $\H$, denoted by $\mathrm{aim}(\H, k)$, as follows
    \[
    \mathrm{aim}(\H, k) \coloneqq \max\{|M| \mid M \text{ is a $k$-admissible matching of $\H$}\}.
    \]
\end{definition}

We note that our definition of $k$-admissible matching is a refinement of \cite[Definition 12]{ErHi1} in the case when $\H$ is a simple graph $G$. Observe that for a simple graph $G$, the condition (1) in \cref{def:aim-d-uniform} is the same as saying that the edges $e_i\in M_i$ and $e_j\in M_j$ form a gap (in the sense of \cite{ErHi1}) for any $i\neq j$.  The only difference, therefore, is that Erey and Hibi considered an additional assumption, namely the induced subgraphs on $M_i$ for each $i\in [r]$ need to be forests. 

The following facts are straightforward to verify; interested readers may look at the analogous results in \cite[Page 6]{ErHi1}. 

\begin{remark}\label{aim property} Let $\H$ be a $d$-uniform hypergraph.
    \begin{enumerate}
    \item A matching $M$ of the hypergraph $\H$ is a $1$-admissible matching if and only if $M$ is an induced matching of $\H$. Thus, $\nu_1(\H) = \mathrm{aim}(\H, 1)$.
    
    \item If $1\leq k \leq  \nu(\H)$, then every $k$-admissible matching of $\H$ is also a $(k+1)$-admissible matching. In particular,
    \[
    \nu_1(\H)=\aim(\H,1)\leq \aim(\H,2)\leq \cdots \leq \aim(\H,\nu(\H))\leq \nu(\H).
    \]
    
    \item For any $1\leq k\leq \nu(\H)$, $\aim(\H, k)\geq k$.

    \item For any $2\leq k\leq \nu(\H)$, $\aim(\H, k)\leq \aim(\H, k-1)+1$.
\end{enumerate}
\end{remark}

\begin{remark}\label{rem:admissable-is-generalized-admissible}
    Let $\H$ be a $d$-uniform hypergraph. Then, one can verify that any $k$-admissible matching $M$ of $\H$ with $\vert M\vert \geq k$ is actually a generalized $k$-admissible matching. For the computation of $\aim(\H,k)$, the following proposition suggests that it is enough to consider only generalized $k$-admissible matching of $\H$. 
\end{remark}

\begin{proposition}\label{prop:aim=lower-bound-for-all-hypergraphs}
    Let $\H$ be a $d$-uniform hypergraph and $1\leq k\leq \nu(\mathcal{H})$. Then
    \[
    (d-1)\aim(\H,k) =  \max\{ |V(M)|-|M|\colon \text{$M$ is a generalized $k$-admissible matching of $\H$}\}.
    \]
\end{proposition}
\begin{proof}
    Note that if $M$ is a $k$-admissible matching of $M$, then $M$ is also a $k$-admissible matching, and $|V(M)|-|M|=(d-1)|M|$. Therefore, by \cref{rem:admissable-is-generalized-admissible}, to prove the equality, it is enough to show that if $M$ is a $k$-admissible matching of $\H$ such that $|M|=\aim(\H,k)$, then  $|M|\ge k$. Evidently, since $k\le \nu(\H)$, there exists a matching $M'$ of $\H$ such that $|M'|=k$. Thus if $|M|<k$, then $|M|<|M'|$, where $M'=M_1'$ is a $k$-admissible matching of $\H$ (by \Cref{def:aim-d-uniform}), a contradiction.
\end{proof}

A straightforward and important consequence of our previous results is that we can use $\aim(\H,k)$ to establish a lower bound for the regularity of square-free powers of $I(\H)$, in the case $\H$ is a $d$-uniform~hypergraph.

\begin{theorem}\label{thm: ordinary lower bound d uniform}
    Let $\H$ be a $d$-uniform hypergraph and $k$ an integer where $1\leq k\leq \nu(\mathcal{H})$. Then 
    \[
    \reg \left( I(\mathcal{H})^{[k]} \right) \geq (d-1)\aim(\H,k) +k.
    \]
    Specifically, for a simple graph $G$ and for any $1\leq k\leq \nu(G)$, we have 
    \[
    \reg \left( I(G)^{[k]}\right) \geq \aim(G,k)+k.
    \]
\end{theorem}
\begin{proof}
    The result follows immediately from \cref{thm: ordinary lower bound} and \cref{prop:aim=lower-bound-for-all-hypergraphs}.
\end{proof}

\begin{remark}
    To distinguish the notion of $k$-admissible matching by Erey and Hibi from the one defined here, we denote their $k$-admissible matching number by $\mathrm{aim}^*(G,k)$ (In \cite[Definition~13]{ErHi1}, they denoted this number by $\mathrm{aim}(G,k)$). For a simple graph $G$, it is immediate from the two definitions that $\mathrm{aim}^*(G,k)\leq \mathrm{aim}(G,k)$. Note that $\aim^*(G,k)=\aim(G,k)$, when $G$ is a forest. However, the difference can be arbitrarily large in the case of arbitrary graphs. For instance, if we consider $G$ to be the complete graph on $2k$ vertices, then $\mathrm{aim}(G,k)=k$, whereas $\mathrm{aim}^*(G,k)=1$.
\end{remark}


\section{Square-free powers of block graphs}\label{sec: 4}

In this section, we consider the family of block graphs and derive a combinatorial formula for the regularity of their square-free powers in terms of the admissible matching numbers. The class of block graphs is a natural generalization of forests. Thus, our formula can be seen as an extension of the regularity formula proved in \cite[Theorem 3.6]{CFL1} for the square-free powers of ~forests.

A \emph{block graph} $G$ is a chordal graph such that any two maximal cliques in $G$ intersect in at most one common vertex. Before going to the proof of the regularity formula, we first introduce some notations and terminologies related to block graphs. Based on these, we are able to establish some structural properties of block graphs that will be helpful later in this section.

\begin{definition}
	Let $G$ be a block graph.
	\begin{enumerate}
		\item[$\bullet$] A maximal clique in $G$ is called a \textit{block}. Let $\Bl$ denote the set of all blocks of $G$.

		\item[$\bullet$] A block $B'\in \Bl$ is said to be a \textit{block neighbor} of $B\in \Bl$ if $|V(B)\cap V(B')|=1$. The set $\{B'\in \Bl\colon |V(B)\cap V(B')|=1\}$ is called the \textit{block neighborhood} of $B$ in $G$, and is denoted by~$\N_G(B)$.
		
		\item[$\bullet$] A block $B\in\Bl$ is said to be a \textit{leaf block} if the number of free vertices in $B$ is at least $|V(B)|-1$.
		
		\item[$\bullet$] A block $B\in \Bl$ is said to be a \textit{distant leaf block} if $B$ is a leaf block and at most one block neighbor of $B$ in $G$ is not a leaf block.
		
		\item[$\bullet$] Let $G$ be a block graph, and $B_1,B_2$ two different blocks in $G$. A \textit{block path} between $B_1$ and $B_2$ of length $n$ is a collection of blocks $\P: D_1,D_2,\ldots, D_n$ satisfying the following ~conditions:
		\begin{enumerate}[label=(\roman*)]
			\item $D_1=B_1$, $D_n=B_2$;
			\item  $D_i\neq D_j$ for each $i\neq j$;
			\item for each $i\neq j$, $V(D_i)\cap V(D_j)\neq\emptyset$ if and only if $j=i-1$ or $j=i+1$.
		\end{enumerate}
	\end{enumerate}
\end{definition}
We explain the above terminologies in the following example.
\begin{example}
    Let $G$ be a block graph as in \Cref{block G}. The block on the vertex set $\{x_{13},x_{14}\}$ is a distant leaf block and its block neighbors are the blocks on the vertex sets $\{x_{13},x_{15}\}$ and $\{x_{10},x_{11},x_{12},x_{13}\}$. Note that the block on the vertex set $\{x_6,x_9\}$ is a leaf block but not a distant leaf block. The block path $\P:D_1,D_2,D_3$ is the (unique) block path between $D_1$ and $D_3$, where $V(D_1)=\{x_1,x_2,x_3\}$, $V(D_2)=\{x_3,x_4,x_5,x_6\}$ and $V(D_3)=\{x_6,x_{10}\}$.
\end{example}

\begin{figure}[h!]
        \centering
         \begin{tikzpicture}
                [scale=1]
            \draw [fill] (-1.2,2) circle [radius=0.08];
            \draw [fill] (-1.2,0) circle [radius=0.08];
            
            \draw [fill] (0,1) circle [radius=0.08];
            \draw [fill] (1,2) circle [radius=0.08];
            \draw [fill] (1,0) circle [radius=0.08];
            \draw [fill] (2,1) circle [radius=0.08];

            \draw [fill] (2.3,2.2) circle [radius=0.08];
            \draw [fill] (3,1.8) circle [radius=0.08];

            \draw [fill] (2.5,0) circle [radius=0.08];

            \draw [fill] (4,1) circle [radius=0.08];
            \draw [fill] (6,1) circle [radius=0.08];
            \draw [fill] (5,2) circle [radius=0.08];
            \draw [fill] (5,0) circle [radius=0.08];
            \draw [fill] (7,2) circle [radius=0.08];
            \draw [fill] (7,0) circle [radius=0.08];
            \draw [fill] (9,1) circle [radius=0.08];
            \draw [fill] (10,1) circle [radius=0.08];

            \node at (-1.2,2.3) {$x_1$};
            \node at (-1.2,-0.3) {$x_2$};
            \node at (-0.5,1) {$x_3$};
            \node at (1,2.5) {$x_4$};
            \node at (1,-0.5) {$x_5$};
            \node at (2,0.5) {$x_6$};
            \node at (2.3,2.5) {$x_7$};
            \node at (3,2.2) {$x_8$};
            \node at (2.5,-0.3) {$x_9$};

            \node at (4,1.5) {$x_{10}$};
            \node at (5,2.5) {$x_{11}$};
            \node at (5,-0.5) {$x_{12}$};
            \node at (6,1.5) {$x_{13}$};
            \node at (7,2.4) {$x_{14}$};
            \node at (7,-0.5) {$x_{15}$};
            \node at (8.8,1.5) {$x_{16}$};
            \node at (10.2,1.5) {$x_{17}$};

            \draw (0,1)--(-1.2,2)--(-1.2,0)--(0,1);
            \draw (0,1)--(1,2)--(1,0)--(2,1)--(0,1)--(1,0);
            \draw (2,1)--(1,2);
            \draw (2,1)--(4,1)--(5,0)--(6,1)--(5,2)--(4,1)--(6,1)--(7,2);
            \draw (5,2)--(5,0);
            \draw (6,1)--(7,0);
            \draw (2.3,2.2)--(2,1)--(3,1.8)--(2.3,2.2);
            \draw (2.5,0)--(2,1);
            \draw (9,1)--(10,1);
\end{tikzpicture}
        \caption{A block graph $G$.}
        \label{block G}
    \end{figure}
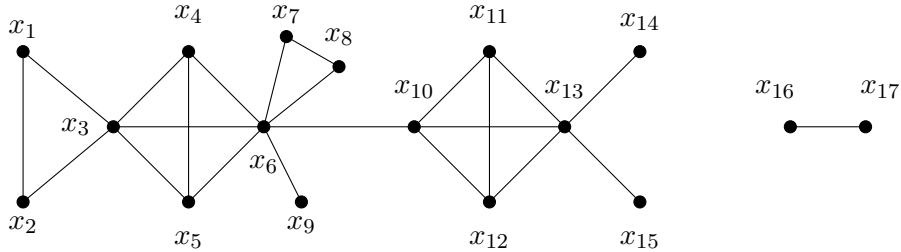

Here, we remark that similar to trees, there exists a (unique) block path between any two blocks in a connected block graph, and there always exists a distant leaf block. These facts might be well-known in the literature on block graphs, but we provide proofs below for the sake of completeness.
\begin{lemma}\label{block path existence}
	Let $G$ be a connected block graph. Then, between any two different blocks in $G$, there is a unique block path.
\end{lemma}
\begin{proof}
	Let $B,B'$ be two different blocks in $G$. Let
	\[
	\Sigma=\{ P\colon a_1,\ldots,a_n\mid a_1\in V(B),a_n\in V(B'),P \text{ is an induced path in }G\}.
	\]
	Choose some $P\colon a_1,\ldots,a_n\in\Sigma$ of the shortest length. Then for each $i\in [n-1]$, there exist distinct $D_{i+1}\in \Bl$ such that $D_{i+1}\neq D_i$ for each $i$, and $\{a_i,a_{i+1}\}\in E(D_{i+1})$. Thus, the collection of blocks $B_1=D_1,D_2,\ldots,D_{n-1},D_n=B'$ satisfy conditions (i) and (ii) in the definition of block path. Moreover, if condition (iii) is not satisfied, then $P$ cannot be of the shortest length. Thus, we have a block path between $B$ and $B'$.

	To prove the uniqueness of such a block path, suppose there are two block paths $\P:D_1,D_2,\ldots,D_n$ and $\P': D_1',D_2',\ldots, D_m'$, where $D_1=D_1'=B$ and $D_n=D_m'=B'$. If $n=2$, then using condition (iii) in the definition of block paths, we see that $m=2$, and thus $\P=\P'$. Therefore, we may assume that $m,n\ge 3$. Let $V(D_i)\cap V(D_{i+1})=\{x_i\}$ for each $i\in [n-1]$ and $V(D_j')\cap V(D_{j+1}')=\{y_j\}$ for each $j\in[m-1]$. First, consider the case when $x_1\neq y_1$. Suppose $2\le i\le n$ and $2\le j\le m$ are the smallest possible integers such that $V(D_i)\cap V(D_j')\neq\emptyset$. Let $V(D_i)\cap V(D_j')=\{z\}$ for some $z\in V(G)$. Then the vertices in $W=\{x_1,\ldots,x_{i-1},z,y_{j-1},\ldots,y_{1}\}$ forms a (not necessarily induced) cycle of length at least $3$ in $G$ such that each edge of this cycle belongs to different blocks. Let $H$ denote the induced subgraph on $W$. Then $H$ contains an induced cycle $C_t$. Since $G$ is chordal, we must have $t=3$. Observe that at least two edges of this cycle belong to different blocks, and hence, each edge is in a different block. Consequently, the block containing this cycle intersects a different block along an edge, a contradiction to the fact that $G$ is a block graph. Thus, we have $\P=\P'$. The case $x_1=y_1$ can be dealt in a similar ~way.
\end{proof}

\begin{lemma}\label{distant leaf block existence}
	Let $G$ be a block graph. Then, $G$ contains a distant leaf block.
\end{lemma}
\begin{proof}
	Let $\P\colon D_1,D_2,\ldots,D_n$ be a block path of the maximum possible length in $G$.  We proceed to show that $D_1$ is a distant leaf block in $G$. If $n=1$, then $G$ is a disjoint union of complete graphs, and it is easy to see that $D_1=D_n$ is a distant leaf block in $G$. Now, suppose $n\ge 2$ and $V(D_j)\cap V(D_{j+1})=\{x_j\}$ for each $j\ge 1$. Let $V(D_1)=\{x_1,y_1,\ldots,y_m\}$ for some $m\ge 1$. Our aim now is to show that $y_i$ is a free vertex in $G$ for each $i\ge 1$. On the contrary, suppose $y_1$ is not a free vertex in $G$. Then there exists some $D\in\Bl$ such that $D\neq D_j$ for each $j\in[n]$ and $V(D)\cap V(D_1)=\{y_1\}$. We proceed to show that $V(D)\cap V(D_j)=\emptyset$ for each $j>1$. If not, suppose $j>1$ is the smallest integer such that $V(D)\cap V(D_j)=\{z\}$ for some $z\in V(G)$. Note that $z\neq x_{i-1}$ since $|V(D)\cap V(D_1)|=1$, and $V(D)\cap V(D_{t})=\emptyset$ for each $t<j$. Now the vertices in $W=\{y_1,x_1,\ldots,x_{i-1},z\}$ form a (not necessarily induced) cycle of length at least $3$ in $G$ such that each edge of this cycle belongs to a different block. Thus, proceeding as in the proof of \Cref{block path existence}, we arrive at a contradiction, and hence, $V(D)\cap V(D_i)=\emptyset$ for each $i>1$. Consequently, $\P'\colon D,D_1,D_2,\ldots,D_n$ is a block path of length $n+1$ in $G$, a contradiction to the fact that $\P$ is a block path of maximum possible length in $G$. Hence, $y_1$ must be a free vertex of $G$, and consequently, $D_1$ is a leaf block.
	
	To prove $D_1$ is a distant leaf block, let us consider $B\in \N_G(D_1)$ such that $B\neq D_2$. In this case, we have $V(B)\cap V(D_1)=\{x_1\}$. Let $V(B)=\{x_1,z_1,\ldots,z_r\}$ for some $r\ge 1$. Proceeding as in the previous paragraph, we see that $V(B)\cap V(D_i)=\emptyset$ for each $i>2$. Thus, $\P''\colon B,D_2,\ldots,D_n$ is also a block path of maximum possible length in $G$. Then, from the arguments in the previous paragraph, we see that $z_r$ is a free vertex in $G$ for each $r\ge 1$. Hence, $B$ is a leaf block, and consequently, $D_1$ is a distant leaf block in $G$. 
\end{proof}

\begin{lemma}\label{more distant leaf block}
	Let $\P:D_1,D_2,\ldots,D_n$ be a block path in a block graph $G$ of the maximum possible length. Let $B\in\N_G(D_2)\setminus
	\{D_3\}$ such that $V(B)\cap V(D_2)\neq V(D_2)\cap V(D_3)$. Then, $B$ is a distant leaf block in~$G$.
\end{lemma}
\begin{proof}
	Let $V(D_i)\cap V(D_{i+1})=\{x_i\}$ for each $i\ge 2$. Let $V(B)\cap V(D_2)=\{z\}$. By the given hypothesis, we have $z\neq x_2$. Arguing as in \Cref{distant leaf block existence}, we see that $V(B)\cap V(D_i)=\emptyset$ for each $i>2$. Hence, $\P':B,D_2,\ldots,D_n$ is a block path of the maximum possible length in $G$. Therefore, by the proof of \Cref{distant leaf block existence}, we see that $B$ is a distant leaf block in $G$.
\end{proof}

Some analogs of the next two results can be obtained for any arbitrary graph, but here, we state and prove them for block graphs only.

\begin{proposition}\label{admissible result}
	Let $G$ be a block graph and $B\in\Bl$. Let $\{u_1,\ldots,u_l\}\subseteq V(B)$ for some $l\ge 1$, and $H$ denote the induced subgraph $G\setminus\{u_1,\ldots,u_l\}$. Then for each $2\le k\le \nu(G)$, we have $\aim(H,k-1)\le\aim(G,k)-1$ if one of the following holds.
	\begin{enumerate}[label=(\roman*)]
		\item For some $i\in[l]$, $|\N_G[u_i]\setminus V(B)|\ge 1$, and for each $v\in \N_G[u_i]\setminus V(B)$, $\deg(v)=1$.
		\item $l\ge 2$, and $u_1,u_2$ are two distinct free vertices in $G$. 
	\end{enumerate}
\end{proposition}
\begin{proof}
	First, let us assume that condition (i) holds. Let $\N_G[u_i]\setminus V(B)=\{v_1,\ldots,v_r\}$ for some $r\ge 1$. Let $M=\sqcup_{i=1}^rM_i$ be a $(k-1)$-admissible matching of $H$ such that $|M|=\aim(H,k-1)$. Since $B$ is a block in $G$, we see that $n\coloneqq|\{i\mid V(M_i)\cap V(B)\neq\emptyset\}|\le 1$. If $n=0$, then it is easy to see that $M'=\sqcup_{i=1}^{r+1}M_i'$, where $M_i'=M_i$ for each $i\in[r]$, and $M_{r+1}'=\{\{u_i,v_1\}\}$, forms a $k$-admissible matching of $G$. Now, suppose $n=1$ and without loss of generality, let $V(M_r)\cap V(B)\neq\emptyset$. In this case, $M'=\sqcup_{i=1}^{r}M_i'$, where $M_i'=M_i$ for each $i\in[r-1]$, and $M_r'=M_r\cup\{\{u_i,v_1\}\}$, forms a $k$-admissible matching of $G$. Thus, combining both cases, we obtain $\aim(H,k-1)\le \aim(G,k)-1$.
	
	For (ii), working with the edge $\{u_1,u_2\}$ similar to $\{u_i,v_1\}$ above and proceeding accordingly we also obtain $\aim(H,k-1)\le \aim(G,k)-1$. This completes the proof.
\end{proof}

\begin{proposition}\label{admissible result1}
	Let $G$ be a block graph and $B\in\Bl$ such that $u_1,u_2\in V(B)$ are two distinct free vertices in $G$. If $H=G\setminus V(B)$, then for each $1\le k\le \nu(G)$, we have $\aim(H,k)\le \aim(G,k)-1$.
\end{proposition}
\begin{proof}
	Let $M=\sqcup_{i=1}^rM_i$ be a $k$-admissible matching of $H$ such that $|M|=\aim(H,k)$. Then $M'=\sqcup_{i=1}^{r+1}M_i'$, where $M_i'=M_i$ for each $i\in[r]$, and $M_{r+1}'=\{\{u_1,u_2\}\}$, forms a $k$-admissible matching of $G$. Thus $\aim(H,k)\le \aim(G,k)-1$.
\end{proof}

Now, we will introduce the notion of a special block in a block graph, which plays a crucial role in proving our main theorem in this section. Let $G$ be a block graph with $B\in\Bl$ and $u\in V(B)$. Define $\N_G(B,u)\coloneqq\{D\in\Bl\mid V(D)\cap V(B)=\{u\}\}$.

\begin{definition}\label{special block}
	Let $G$ be a block graph, and $B$ a block of $G$ with $V(B)=\{u_1,\ldots,u_d\}$. We say that $B$ is a \emph{special block} of $G$ if $B$ is one of the following types.
	\begin{enumerate}[leftmargin=2cm]
		\item[{\bf Type I}] $d\le 2$, and $\N_G(B,u_i)=\emptyset$ for each $i\in[d]$.
		\item[{\bf Type II}] $d\ge 3$, and $\N_G(B,u_i)=\emptyset$ for each $i\in[d-1]$.
		\item[{\bf Type III}] $d\ge 2$, and there exists some $i\in[d-1]$ such that $\N_G(B,u_i)\neq\emptyset$. Moreover, if $\N_G(B,u_i)\neq\emptyset$ for some $i\in[d-1]$, then for each $D\in \N_G(B,u_i)$, $D\cong K_2$.
	\end{enumerate}
\end{definition}

\begin{example}
    Let $G$ be a block graph as in \Cref{block G}. Then the block on the vertex set $\{x_{16},x_{17}\}$ is a special block of Type I whereas the block on the vertex set $\{x_1,x_2,x_3\}$ is a special block of Type II. On the other hand, the block on the vertex set $\{x_{10},x_{11},x_{12},x_{13}\}$ is a special block of Type III. 
\end{example}

\begin{lemma}\label{special block existence lemma}
	Let $G$ be a block graph. Then, $G$ contains at least one special block.
\end{lemma}
\begin{proof}
	First, consider the case when $G$ contains an isolated vertex $\{u\}$. In this case, $B=G[\{u\}]$ is a special block of Type I in $G$. Next, suppose $G$ has no isolated vertices. Let
	\[
	\Sigma=\{B\in \Bl\mid \P:B=D_1,D_2,\ldots,D_n\text{ is a block path of maximum possible length in }G\}.
	\]
	Then, by the proof of \cref{distant leaf block existence}, we see that $B$ is a distant leaf block of $G$. Observe that for each $B\in\Sigma$, $|V(B)|\ge 2$. If $|V(B)|\ge 3$, for some $B\in\Sigma$, then $B$ is a special block of Type II in ~$G$.
	
	Next, suppose for each $B\in\Sigma$, we have $|V(B)|=2$. Fix some $B\in\Sigma$ and let $\P:B=D_1,D_2,\ldots,D_n$ be a block path of the maximum possible length in $G$. If $n=1$, then $G$ is a disjoint union of edges, and thus $B$ is a special block of Type I in $G$.
	
	Now consider the case when $n\ge 2$, i.e., $\P$ is a block path of length at least $2$. Let $V(D_2)=\{u_1,\ldots,u_d\}$, where $V(D_1)\cap V(D_2)=\{u_1\}$ and $V(D_2)\cap V(D_3)=\{u_d\}$. Here, we remark that the value of $n$ can also be $2$, and in that case, no such $D_3$ exists. Without loss of generality, let $\N_G(D_2)=\{\E_{i,j}\mid i\in[d],j\le r_i\}$, where $r_i$ is a non-negative integer for each $i\in[d]$, $V(\E_{i,j})\cap V(D_2)=\{u_i\}$, $\E_{1,1}=D_1$, and $\E_{d,1}=D_3$. If $r_i=0$ for some $i\in[d]$, then we make the convention that $u_i$ is a free vertex in $G$ so that no such $\E_{i,j}$ exists. Thus, we always have $r_1\ge 1$. Since $\P$ is a block path of maximum length in $G$, by \Cref{more distant leaf block}, we have that $\E_{i,j}$ is a distant leaf block in $G$ for each $i\in[d-1]$ and $j\in[r_i]$. Moreover, $\E_{i,j}\in\Sigma$ for such $i$ and $j$. Then by our assumption $|V(\E_{i,j})|=2$ for each $i\in[d-1]$ and $j\in[r_i]$, i.e., $\E_{i,j}\cong K_2$. Consequently, $D_2$ is a special block of Type III in $G$, and this completes the proof.
\end{proof}

\begin{notation}\label{extra ideal}
	Let $B$ be a special block in a block graph $G$. Let $V(B)=\{u_1,\ldots,u_d\}$, and for each $i\in[d-1]$, either $\N_G(B,u_i)=\emptyset$ or if for some $i\in[d-1]$, $\N_G(B,u_i)\neq\emptyset$, then for each $D\in \N_G(B,u_i)$, we have $D\cong K_2$. Define
	\[
	\Lambda\coloneqq\cup_{i=1}^{d-1}\N_G(B,u_i).
	\]
	Then $\Lambda$ is a collection of blocks in $G$ which are isomorphic to $K_2$. Now, for each $S\subseteq\Lambda$, we define the ideal
	\[
	I_{S,B}=\l xy\mid x,y\in V(D) \text{ for some }D\in S\r.
	\]
\end{notation}

\begin{theorem}\label{block ordinary upper bound}
	Let $G$ be a block graph on $n$ vertices. Then, for each $1\le k\le \nu(G)$,
	\begin{enumerate}[label=(\roman*)]
		\item $\reg(I(G)^{[k]})\le \aim(G,k)+k$.
		\item $\reg(I(G)^{[k]}+I_{S,B})\le\aim(G,k)+k$, where $B$ is a special block in $G$, and $I_{S,B}$ is the ideal associated to any $S\subseteq \Lambda$ as described in \Cref{extra ideal}.
	\end{enumerate}
	Consequently, for a block graph $G$, we have $\reg(I(G)^{[k]})=\aim(G,k)+k$.
\end{theorem}
\begin{proof}
	We prove (1) and (2) simultaneously by induction on $n$. By \Cref{special block existence lemma}, $G$ contains at least one special block. Let $B$ be a special block in $G$, and $I_{S,B}$ is the ideal defined as in \Cref{extra ideal}, where $S\subseteq \Lambda$. Now, if $n\le 3$, then $k=1$, and thus it easy to see that 
	\[
	\reg(I(G)^{[1]}+I_{S,B})=\reg(I(G)^{[1]})=\nu(G)+1=\aim(G,1)+1.
	\]
	Therefore, we may assume that $n\ge 4$. For such an $n$, the case $k=1$ follows from \cite[Theorem 6.8]{HaVanTuyl2008} since $I(G)^{[1]}+I_{S,B}=I(G)^{[1]}$ and $\nu(G)=\aim(G,1)$. Therefore, we can assume $k\ge 2$.
	
	\noindent
	\textbf{Case I:} $B$ is a special block of Type I in $G$. In this case, $B$ is a connected component of $G$, and either $B$ is an isolated vertex in $G$ or $B\cong K_2$. In both cases, we see that $\Lambda=\emptyset$, and $I_{S,B}=\l 0\r$. Now if $B$ is an isolated vertex of $G$, then by the induction hypothesis and by \Cref{admissible result2}, we have
	\begin{align*}
		\reg(I(G)^{[k]})&=\reg(I(G\setminus V(B))^{[k]})\\
		&\le \aim(G\setminus V(B),k)+k\\
		&\le\aim(G,k)+k.
	\end{align*}
	Now, suppose $B\cong K_2$ and $V(B)=\{u,v\}$. Then,
	\begin{equation}\label{eq21}
		\begin{split}
			\reg((I(G)^{[k]}+\l uv\r):u)&=\reg(I(G\setminus u)^{[k]}+\l v\r)\\
			&\le \aim(G\setminus u,k)+k\\
			&\le\aim(G,k)+k-1,
		\end{split}
	\end{equation}
	where the equality follows from \cite[Lemma 2.9]{DRS20242}, and the first and second inequalities follow from the induction hypothesis and \Cref{admissible result2}, respectively. Next,
	\begin{equation}\label{eq22}
		\begin{split}
			\reg(I(G)^{[k]}+\l uv\r+\l u\r)&=\reg((I(G\setminus u)^{[k]}+\l u\r)\\
			&\le\aim(G\setminus u,k)+k\\
			&\le \aim(G,k)+k
		\end{split}
	\end{equation}
	where, as before, the equality follows from \cite[Lemma 2.9]{DRS20242}, and the first and second inequalities follow from the induction hypothesis and \Cref{admissible result2}, respectively. Now using inequalities (\ref{eq21}), (\ref{eq22}), and \Cref{regularity lemma} together, we obtain
	\begin{align}\label{eq23}
		\reg(I(G)^{[k]}+\l uv\r)\le\aim(G,k)+k.
	\end{align}
	Next, by \cite[Lemma 2.9]{DRS20242}, $(I(G)^{[k]}:uv)=I(G\setminus\{u,v\})^{[k-1]}$. Observe that both $u$ and $v$ are free vertices of $G$. Therefore, using the induction hypothesis, and \Cref{admissible result}, we obtain 
	\begin{equation}\label{eq24}
		\begin{split}
			\reg(I(G)^{[k]}:uv)&=\reg(I(G\setminus\{u,v\})^{[k-1]})\\
			&\le\aim(G\setminus\{u,v\},k-1)+k-1\\
			&\le\aim(G,k)+k-2.
		\end{split}
	\end{equation}
	Combining inequalities (\ref{eq23}) and (\ref{eq24}), and using \Cref{regularity lemma}, we get $\reg(I(G)^{[k]})\le\aim(G,k)+k$. This completes the proof of (ii) as well as (i).   
	
	\noindent
	\textbf{Case II:} $B$ is a special block of Type II in $G$. Hence, $|V(B)|\ge 3$. As in Case I, in this case too, we have $\Lambda=\emptyset$, and $I_{S,B}=\l 0\r$. Let $V(B)=\{u_1,\ldots,u_d\}$, where $d\ge 3$, and $u_1,\ldots,u_{d-1}$ are free vertices in $G$. 
	
	\noindent
	\textbf{Claim 1}: For each $j\in[d-1]$, 
	\[
	\reg((I(G)^{[k]}+\l u_1u_2,\ldots,u_1u_j\r):u_1u_{j+1})\le \aim(G,k)+k-2,
	\]
	where $(I(G)^{[k]}+\l u_1u_2,\ldots,u_1u_j\r=I(G)^{[k]}$ in case $j=1$.
	
	\noindent
	\textit{Proof of Claim 1}: Let $G_j=G\setminus\{u_1,\ldots,u_{j+1}\}$. Then, 
	\begin{align*}
		\reg((I(G)^{[k]}+\l u_1u_2,\ldots,u_1u_j\r):u_1u_{j+1})&=\reg(I(G_j)^{[k-1]}+\l u_2,\ldots,u_j\r)\\
		&\le \aim(G_j,k-1)+k-1\\
		&\le \aim(G,k)+k-2,
	\end{align*}
	where the first equality follows from \cite[Lemma 2.9]{DRS20242}, and the first and second inequalities follow from the induction hypothesis and \Cref{admissible result}, respectively. This completes the proof of ~Claim~1.
	
	Let $H=G\setminus N_G[u_1]$. Then,
	\begin{equation}\label{eq27}
		\begin{split}
			\reg((I(G)^{[k]}+\l u_1u_2,\ldots,u_1u_d\r):u_1)&=\reg(I(H)^{[k]}+\l u_2,\ldots,u_d\r)\\
			&\le \aim(H,k)+k\\
			&\le \aim(G,k)+k-1,
		\end{split}
	\end{equation}
	where the first equality again follows from \cite[Lemma 2.9]{DRS20242}, and the first and second inequalities follow from the induction hypothesis and \Cref{admissible result1}, respectively. Next, we have
	\begin{equation}\label{eq28}
		\begin{split}
			\reg(I(G)^{[k]}+\l u_1u_2,\ldots,u_1u_d,u_1\r)&=\reg(I(G\setminus u_1)^{[k]})\\
			&\le\aim(G\setminus u_1,k)+k\\
			&\le \aim(G,k)+k,
		\end{split}
	\end{equation}
	where the first and second inequalities follow from the induction hypothesis, and \Cref{admissible result2},~respectively.
	
	Now, by \Cref{regularity lemma} along with inequalities~(\ref{eq27}) and (\ref{eq28}), we get $\reg(I(G)^{[k]}+\l u_1u_2,\ldots,u_1u_d\r)\le\aim(G,k)+k$. Observe that from Claim 1, we have $\reg((I(G)^{[k]}+\l u_1u_2,\ldots,u_1u_{d-1}\r):u_1u_d)\le\aim(G,k)+k-2$, and hence using \Cref{regularity lemma} again we obtain $\reg(I(G)^{[k]}+\l u_1u_2,\ldots,u_1u_{d-1}\r)\le\aim(G,k)+k$. By a repeated use of Claim 1 and \Cref{regularity lemma}, we finally get $\reg(I(G)^{[k]})\le\aim(G,k)+k$, as desired.
	
	\noindent
	\textbf{Case III:} $B$ is a special block of Type III in $G$. Let $V(B)=\{u_1,\ldots,u_d\}$, and without loss of generality, let $\N_G(B,u_1)\neq\emptyset$. Now, for each $i\in[d-1]$, let $\N_G(B,u_i)=\{\E_{i,j}\mid j\le r_i\}$, where $r_i$ is a non-negative integer. As mentioned in the proof of \Cref{special block existence lemma}, if $r_i=0$ for some $i\in[d]$, then we make the convention that $u_i$ is a free vertex in $G$. Thus, we always have $r_1\ge 1$. By our assumption, we also have $\E_{i,j}\cong K_2$. If such an $\E_{i,j}$ exists, let $V(\E_{i,j})=\{u_i,v_{i,j}\}$ for each $i\in[d-1]$ and $j\in [r_i]$. In this case, our first claim is the following.
	
	\noindent
	\textbf{Claim 2:} $\reg(I(G)^{[k]}+I_{\Lambda,B})\le\aim(G,k)+k$.
	
	\noindent
	\textit{Proof of Claim 2}: By \cite[Lemma 2.9]{DRS20242}, for each $p\in[d-1]$, we have
	\begin{equation}\label{eq29}
		\begin{split}
			&((I(G)^{[k]}+ I_{\Lambda,B}+\l u_1u_2,\ldots, u_1u_p\r):u_1u_{p+1})\\
			&=I(G\setminus \{u_1,\ldots,u_p,u_{p+1}\})^{[k-1]}+\l u_iv_{i,j}\mid p+2\le i\le d-1\r\\
			&\hspace{4cm}+\l v_{1,j},v_{p+1,l}\mid j\in [r_1],p+1\le d-1, l\in[r_{p+1}]\r.
		\end{split}
	\end{equation}
	Here we make the convention that $I(G)^{[k]}+I_{\Lambda,B}+\l u_1u_2,\ldots, u_1u_p\r=I(G)^{[k]}+I_{\Lambda,B}$, in case $p=1$. Note that, in \Cref{eq29}, we have $\l u_iv_{i,j}\mid p+2\le i\le d-1\r=\l 0\r$ if $p\in\{d-2,d-1\}$. Moreover, if $p=d-3$, and $r_{d-1}=0$, then we again have $\l u_iv_{i,j}\mid p+2\le i\le d-1\r=\l 0\r$. In all other cases, let $B'=B\setminus\{u_1,\ldots,u_{p+1}\}$. Then it is easy to see that $B'$ is a special block in $G\setminus \{u_1,\ldots,u_{p+1}\}$. Moreover, $I_{\Lambda',B'}=\l u_iv_{i,j}\mid p+2\le i\le d-1\r$, where $\Lambda'=\cup_{i=p+2}^{d-1}\N_{G\setminus \{u_1,\ldots,u_p,u_{p+1}\}}(B',u_i)$. Now, by the induction hypothesis and using \Cref{admissible result}, we obtain
	\begin{equation}\label{eq4}
		\begin{split}
			\reg((I(G)^{[k]}+I_{\Lambda,B}+\l u_1u_2,\ldots, u_1u_p\r):u_1u_{p+1})&\le \aim(G\setminus \{u_1,\ldots,u_{p+1}\},k-1)+k-1\\
			&\le \aim(G,k)+k-2.
		\end{split}
	\end{equation}
	In particular,
	\begin{align}\label{eq3}
		\reg((I(G)^{[k]}+I_{\Lambda,B}+\l u_1u_2,\ldots, u_1u_{d-1}\r):u_1u_{d})\le \aim(G,k)+k-2.
	\end{align}
	Next, we have 
	\begin{equation}\label{eq1}
		\begin{split}
			\reg((I(G)^{[k]}+I_{\Lambda,B}+\l u_1u_2,\ldots, u_1u_{d}\r):u_1)&=\reg(I(G\setminus \{u_1,\ldots,u_d\})^{[k]})\\
			&\le \aim(G\setminus \{u_1,\ldots,u_d\},k)+k\\
			&\le \aim(G,k)+k-1,
		\end{split}
	\end{equation}
	where the equality follows from \cite[Lemma 2.9]{DRS20242}, and the first and second inequalities follow from the induction hypothesis and \Cref{admissible result1}, respectively. Moreover, $I(G)^{[k]}+I_{\Lambda,B}+\l u_1u_2,\ldots, u_1u_{d}\r+\l u_1\r=I(G\setminus u_1)^{[k]}+\l u_iv_{i,j}\mid 2\le i\le d-1, j\in[r_i]\r$. As before, the ideal $\l u_iv_{i,j}\mid 2\le i\le d-1, j\in[r_i]\r=\l 0\r$ if $d=2$, or $d=3$ with $r_2=0$. In all other cases, let $B''=B\setminus u_1$. Then, it is easy to see that $B''$ is a special block in $G\setminus u_1$. Moreover, $I_{\Lambda'',B''}=\l u_iv_{i,j}\mid 2\le i\le d-1, j\in[r_i]\r$, where $\Lambda''=\cup_{i=2}^{d-1}\N_{G\setminus u_1}(B'',u_i)$. Therefore, by the induction hypothesis and by \Cref{admissible result2}, we get
	\begin{equation}\label{eq2}
		\begin{split}
			\reg(I(G)^{[k]}+I_{\Lambda,B}+\l u_1u_2,\ldots, u_1u_{d}\r+\l u_1\r)&\le \aim(G\setminus u_1,k)+k\\
			&\le \aim(G,k)+k.
		\end{split}
	\end{equation}
	In view of \Cref{regularity lemma}, inequalities~(\ref{eq1}) and (\ref{eq2}) give
	\begin{align*}
		\reg(I(G)^{[k]}+I_{\Lambda,B}+\l u_1u_2,\ldots, u_1u_{d}\r\r)\le \aim(G,k)+k.
	\end{align*}
	On the other hand, using inequality (\ref{eq3}) and \Cref{regularity lemma} again, we obtain
	\[
	\reg(I(G)^{[k]}+I_{\Lambda,B}+\l u_1u_2,\ldots, u_1u_{d-1}\r\r)\le \aim(G,k)+k.
	\]
	Thus, by a repeated use of inequality (\ref{eq4}) and \Cref{regularity lemma}, we have $\reg(I(G)^{[k]}+I_{\Lambda, B})\le \aim(G,k)+k$, and this completes the proof of Claim 2.
	
	\noindent
	\textbf{Claim 3}: $
	\reg(I(G)^{[k]}+ I_{S,B})\le \aim(G,k)+k
	$, where $S\subseteq \Lambda=\{u_iv_{i,j}\mid i\in[d-1]\text{ and }j\in [r_i]\}$.
	
	\noindent
	\textit{Proof of Claim 3}: We induct on $|\Lambda|-|S|$. When $S=\Lambda$, it follows from Claim 2. Let $|S|<|\Lambda|$, and choose some $u_lv_{l,t}\in \Lambda\setminus S$. Then $l\in [d-1]$, and $t\in[r_l]$. Let $S_1=S\cup \E_{l,t}$. By the induction hypothesis, we have 
	\begin{equation}\label{eq10}
		\begin{split}
			\reg(I(G)^{[k]}+I_{S,B}+\l u_lv_{l,t}\r)&=\reg(I(G)^{[k]}+I_{S_1,B})\\
			&\le \aim(G,k)+k.
		\end{split}
	\end{equation}
	
	Now, by \cite[Lemma 2.9]{DRS20242}, $(I(G)^{[k]}+I_{S,B}): u_lv_{l,t})=I(\Gamma)^{[k-1]}+\l u_iv_{i,j}\mid \E_{i,j}\in S\text{ with }i\in[d-1]\setminus\{l\}\text{ and }j\in[r_i]\r+\l v_{l,q}\mid u_lv_{l,q}\in I_{S,B}\text{ and }q\neq t\r$, where $\Gamma=G\setminus \{u_l\}$. Again, observe that if $d=2$, or $d=3$ with $\E_{i,j}\notin S$ for each $i\in[d-1]\setminus\{l\}$ and $j\in [r_i]$, then the ideal $\l u_iv_{i,j}\mid \E_{i,j}\in S\text{ with }i\in[d-1]\setminus\{l\}\text{ and }j\in[r_i]\r=\l 0\r$. In all other cases, let $B'''=B\setminus u_l$. Then we see that $B'''$ is a special block in $G\setminus u_l$. Moreover, $I_{S_1,B'''}=\l u_iv_{i,j}\mid \E_{i,j}\in S\text{ with }i\in[d-1]\setminus\{l\}\text{ and }j\in[r_i]\r$, where $S_1=S\cap (\cup_{\underset{i\neq l}{i=1}}^{d-1}\N_{G\setminus u_l}(B''',u_i))$. Thus, by the induction hypothesis and by \Cref{admissible result}, we have
	\begin{align*}
		\reg((I(G)^{[k]}+I_{S,B}): u_lv_{l,t})&\le \aim(\Gamma,k-1)+k-1\\
		&\le \aim(G,k)+k-2.
	\end{align*}
	Therefore, using inequality~(\ref{eq10}) and \Cref{regularity lemma} we obtain
	\begin{align}\label{eq11}
		\reg((I(G)^{[k]}+I_{S,B})\le\aim(G,k)+k.
	\end{align}
	This completes the proof of (ii). Now taking $S=\emptyset$ in inequality~(\ref{eq11}), we  obtain $\reg((I(G)^{[k]})\le \aim(G,k)+k$, and this proves (i).
	
	Finally, the last statement follows immediately from (i) and \Cref{thm: ordinary lower bound d uniform}. This completes the proof of the theorem.
\end{proof}

\section{Second square-free power of Cohen-Macaulay chordal graphs}\label{sec: 5}
   
   In this section, we consider the second square-free power of the edge ideal of a Cohen-Macaulay chordal graph $G$ and show that $\mathrm{reg}(I(G)^{[2]})=\mathrm{aim}(G,2)+2$ whenever $\nu(G)\ge 2$. Recall that Herzog, Hibi, and Zheng combinatorially classified the Cohen-Macaulay property of the edge ideal of a chordal graph in \cite{CMchordal} and called such a graph Cohen-Macaulay chordal, which we describe below.

   \begin{definition}
        A graph $G$ is said to be a \emph{Cohen-Macaulay chordal} graph if $G$ is chordal and $V(G)$ can be partitioned into $W_1,\ldots,W_m$ such that $G[W_i]$ is a maximal clique containing a free vertex for each $i\in[m]$.
   \end{definition}
   
   Although we are computing the regularity of only the second square-free power of Cohen-Macaulay chordal graphs, the proof is not straightforward. To establish our main result and to make it convenient for the readers, we first prove some auxiliary results. Recall from \cite[Lemma~2.9]{DRS20242} that for any graph $G$ and any $\{x,y\}\in E(G)$, the ideal $(I(G)^{[2]}:xy)$ is an edge ideal of a graph. In this section, we always write,
   \[
   (I(G)^{[2]}:xy)=I(\widetilde{G}),
   \] 
    where $\widetilde{G}$ is this new graph constructed from $G$ with
    \begin{align*}
        V(\widetilde{G})&=V(G)\setminus\{x,y\}\text{ and } \\ E(\widetilde{G})&=E(G\setminus\{x,y\})\cup\{\{u,v\}\mid u\in N_G(x)\setminus\{y\},v\in N_G(y)\setminus\{x\},\text{ and }u\neq v\}.
    \end{align*}

\begin{wrapfigure}{R}{0.25\textwidth}
        \begin{tikzpicture}[scale=0.75]
                \draw [fill] (0,0) circle [radius=0.1];
                \draw [fill] (2,0) circle [radius=0.1];
                \draw [fill] (4,0) circle [radius=0.1];
                \draw [fill] (0,2) circle [radius=0.1];
                \draw [fill] (2,2) circle [radius=0.1];
                \draw [fill] (4,2) circle [radius=0.1];
                \node at (0,-0.5) {$y$};
                \node at (2,-0.5) {$w_4$};
                \node at (4,-0.5) {$w_3$};
                \node at (0,2.5) {$x$};
                \node at (2,2.5) {$w_1$};
                \node at (4,2.5) {$w_2$};                           
                \draw (0,0)--(2,0)--(4,0)--(4,2)--(2,2)--(0,2)--(0,0);
                \draw[dashed, red] (2,0)--(2,2);
                \draw [thick,blue] (0,0)--(4,2);
                \draw [thick,blue] (0,2)--(4,0);
        \end{tikzpicture}
    \end{wrapfigure}
Our first aim is to show that whenever $G$ is a chordal graph and $\{x,y\}\in E(G)$, the graph $\widetilde{G}$ is weakly chordal. To prove this, we need the following lemma.

\begin{lemma}\label{edge lemma}
	Let $G$ be a chordal graph with $\{x,y\}\in E(G)$ and let $(I(G)^{[2]}:xy)=I(\widetilde{G})$. Assume that $\widetilde{G}[\{w_1,w_2,w_3,w_4\}]\cong C_4$ with edges $\{w_1,w_2\}$, $\{w_2,w_3\}$, $\{w_3,w_4\}$, and $\{w_1,w_4\}$. If $\{x,w_1\},\{y,w_4\}\in E(G)$, and $\{w_1,w_4\}\in E(\widetilde{G})\setminus E(G)$, then $\{x,w_3\},\{y,w_2\}\in E(G)$.
\end{lemma}

\noindent
\textit{Proof.} First note that $\{x,w_2\}\notin E(G)$ since $\{w_2,w_4\}\notin E(\widetilde{G})$. Similarly, $\{y,w_3\}\notin E(G)$. Now we proceed to show that $\{x,w_3\}\in E(G)$. On the contrary, suppose $\{x,w_3\}\notin E(G)$. This implies $\{w_2,w_3\},\{w_3,w_4\}\in E(G)$ since we also have $\{y,w_3\}\notin E(G)$. Now if $\{y,w_2\}\in E(G)$, then $G[\{y,w_2,w_3,w_4\}]\cong C_4$, a contradiction to the fact that $G$ is chordal. If $\{y,w_2\}\notin E(G)$, then it is easy to see that $\{w_1,w_2\}\in E(G)$ since $\{x,w_2\}\notin E(G)$. Consequently, $G[\{x,w_1,w_2,w_3,w_4\}]\cong C_5$ if $\{x,w_4\}\in E(G)$, and $G[\{y,w_1,w_2,w_3,w_4\}]\cong C_5$ if $\{y,w_1\}\in E(G)$, and $G[\{x,y,w_1,w_2,w_3,w_4\}]\cong C_6$ if neither $\{x,w_4\},\{y,w_1\}$ are edges of $G$. In all cases, we have a contradiction since $G$ is chordal. Therefore, we must have $\{x,w_3\}\in E(G)$. Proceeding along the same lines as above, we also have $\{y,w_2\}\in E(G)$, by symmetry. 
\qed

Using the above lemma, we show the following important result.

\begin{proposition}\label{weakly chordal}
	Let $\{x,y\}$ be an edge of a chordal graph $G$ and let $(I(G)^{[2]}:xy)=I(\widetilde{G})$. Then, $\widetilde{G}$ is a weakly chordal graph.
\end{proposition}

\noindent
\textit{Proof.}
	First, we proceed to show that $\widetilde{G}$ does not contain any induced cycle of length $5$ or more. On the contrary, suppose $\widetilde{G}$ has an induced cycle $C_n$ on the vertex set $\{u_1,\ldots,u_n\}$ with edge set $\{\{u_i,u_{i+1}\},\{u_1,u_n\}\mid i\in[n-1]\}$, where $n\ge 5$. Without loss of generality we may assume that $\{u_1,u_n\}\in E(\widetilde{G})\setminus E(G)$ since $G$ is chordal. Let us also assume $\{x,u_1\},\{y,u_n\}\in E(G)$. Our aim now is to prove the following:
	\begin{enumerate}[label=(\roman*)]
		\item $\{x,u_i\}\notin E(G)$ for $2\le i\le n-2$,
		\item $\{y,u_j\}\notin E(G)$ for $3\le j\le n-1$,
		\item $\{u_i,u_{i+1}\}\in E(G)$ for $2\le i\le n-2$.
	\end{enumerate}
	Note that if $\{x,u_i\}\in E(G)$ for some $2\le i\le n-2$, then $\{u_i,u_n\}\in E(\widetilde{G})$, which is a contradiction to the fact that $\widetilde{G}[\{u_1,\ldots,u_n\}]\cong C_n$. This proves (i). Similarly, $\{y,u_j\}\in E(G)$ for some $3\le j\le n-1$ implies $\{u_1,u_j\}\in E(\widetilde{G})$, which is again a contradiction. Thus, we have (ii). To prove (iii), observe that if for some $2\le i\le n-3$, $\{u_i,u_{i+1}\}\notin E(G)$, then either $\{x,u_i\}\in E(G)$ or $\{x,u_{i+1}\}\in E(G)$, which is a contradiction to (i) above. Now if $\{u_{n-2},u_{n-1}\}\notin E(G)$, then either $\{x,u_{n-2}\}\in E(G)$ or $\{y,u_{n-2}\}\in E(G)$. But $\{x,u_{n-2}\}\notin E(G)$ and $\{y,u_{n-2}\}\notin E(G)$ because of (i) and (ii), respectively. Thus, we have (iii).

    \begin{wrapfigure}{R}{0.25\textwidth}
        \begin{tikzpicture}[scale=0.75]
            \draw [fill] (0,0) circle [radius=0.1];
            \draw [fill] (2,0) circle [radius=0.1];
            \draw [fill] (4,0) circle [radius=0.1];
            \draw [fill] (0,2) circle [radius=0.1];
            \draw [fill] (2,2) circle [radius=0.1];
            \draw [fill] (4,2) circle [radius=0.1];
            \draw [fill] (2,-2) circle [radius=0.1];
            \draw [fill] (4,-2) circle [radius=0.1];

            \node at (0,-0.5) {$y$};
            \node at (2.7,-0.5) {$v_{n-1}$};
            \node at (4.5,0) {$v_2$};
            \node at (0,2.5) {$x$};
            \node at (2,2.5) {$v_1$};
            \node at (4,2.5) {$v_{n-2}$};
            \node at (2,-2.5) {$v_3$};
            \node at (4,-2.5) {$v_{n}$};

            \draw (0,0)--(2,0)--(4,0)--(4,2)--(2,2)--(0,2)--(0,0);
            \draw[dashed, red] (2,0)--(2,2);
            \draw  (2,0)--(2,-2)--(4,-2)--(4,0);
    \end{tikzpicture}
 \end{wrapfigure}
	Next, we proceed to show that $\{x,u_{n-1}\},\{y,u_2\}\in E(G)$. Indeed, observe that if $\{x,u_{n-1}\}\notin E(G)$, then $\{u_n,u_{n-1}\}\in E(G)$,  by (ii). In this case if $\{y,u_2\}\in E(G)$, then $G[\{y,u_2,\ldots,u_{n}\}]\cong C_n$, where $n\ge 5$. This contradicts the fact that $G$ is chordal. Hence we have $\{x,u_{n-1}\},\{y,u_2\}\notin E(G)$. But this implies $\{u_1,u_2\}\in E(G)$, by (i), and consequently, we have $G[\{x,u_1,u_2,\ldots,u_n\}]\cong C_{n+1}$ if $\{x,u_n\}\in E(G)$; $G[\{y,u_1,u_2,\ldots,u_n\}]\cong C_{n+1}$ if $\{y,u_1\}\in E(G)$;   $G[\{x,y,u_1,u_2,\ldots,u_n\}]\cong C_{n+2}$ if $\{x,u_n\},\{y,u_1\}\notin E(G)$. In all the cases, we again have contradictions. Therefore, we must have $\{x,u_{n-1}\}\in E(G)$. Similarly, one can show that $\{y,u_2\}\in E(G)$. However, $\{x,u_{n-1}\},\{y,u_2\}\in E(G)$ implies $\{u_2,u_{n-1}\}\in E(\widetilde{G})$, a contradiction to the fact that $\widetilde{G}[\{u_1,\ldots,u_n\}]\cong C_n$. This completes the proof of the fact that $\widetilde{G}$ does not contain any induced cycle of length $5$ or more.
	
	Next, our aim is to show that $(\widetilde{G})^c$ does not contain any induced cycle of length $5$ or more. From above, it is easy to see that $(\widetilde{G})^c$ does not contain any induced $C_5$ since the complement of $C_5$ is again a $C_5$. Now, on the contrary, suppose $(\widetilde{G})^c$ has an induced cycle $C_n$ on the vertex set $\{v_1,\ldots,v_n\}$ with edge set $\{\{v_i,v_{i+1}\},\{v_1,v_n\}\mid i\in[n-1]\}$, where $n\ge 6$. Then observe that $\widetilde{G}[\{v_1,v_2,v_{n-2},v_{n-1}\}]\cong C_4$. We have the following four cases:
	
	\noindent
	\textbf{Case I:} $\{v_1,v_{n-1}\}\in E(\widetilde{G})\setminus E(G)$. In this case, without loss of generality, we may assume that $\{x,v_1\},\{y,v_{n-1}\}\in E(G)$. Then, by \Cref{edge lemma}, we have $\{x,v_2\},\{y,v_{n-2}\}\in E(G)$. Now, observe that $\widetilde{G}[\{v_2,v_3,v_{n-1},v_n\}]\cong C_4$ and since $G$ is chordal, the induced subgraph $\widetilde{G}[\{v_2,v_3,v_{n-1},v_n\}]$ of $\widetilde{G}$ must have a new edge. If the new edge is $\{v_3,v_{n-1}\}$, then either $\{y,v_3\}\in E(G)$ or $\{x,v_3\}\in E(G)$. However, $\{y,v_3\}\in E(G)$ implies $\{v_2,v_3\}\in E(\widetilde{G})$, a contradiction to the fact that $(\widetilde{G})^c[v_1,\ldots,v_n]\cong C_n$. Similarly, $\{x,v_3\}\in E(G)$ together with \Cref{edge lemma} imply $\{y,v_n\}\in E(G)$, and consequently, $\{v_1,v_n\}\in E(\widetilde{G})$; again a contradiction. Therefore, $\{v_3,v_{n-1}\}$ cannot be the new edge. On the other hand, $\{x,v_n\}\notin E(G)$ since $\{v_{n-1},v_n\}\notin E(\widetilde{G})$. Similarly, $\{y,v_n\}\notin E(G)$ since $\{v_1,v_n\}\notin E(\widetilde{G})$. Consequently, both $\{v_3,v_n\}$ and $\{v_2,v_n\}$ cannot be the new edge. Lastly, if $\{v_2,v_{n-1}\}$ is the new edge, then by \Cref{edge lemma}, $\{y,v_n\}\in E(G)$, and consequently, $\{v_1,v_n\}\in E(\widetilde{G})$, again a contradiction.
	
	\noindent
	\textbf{Case II:} $\{v_2,v_{n-2}\}\in E(\widetilde{G})\setminus E(G)$. Without loss of generality, assume that $\{x,v_{n-2}\},\{y,v_2\}\in E(G)$. Then, by \Cref{edge lemma}, we have $\{x,v_{n-1}\},\{y,v_{1}\}\in E(G)$. Now, by similar arguments as before, we have $\widetilde{G}[\{v_2,v_3,v_{n-1},v_n\}]\cong C_4$ and thus the induced subgraph $\widetilde{G}[\{v_2,v_3,v_{n-1},v_n\}]$ of $\widetilde{G}$ must have a new edge since  $G$ is chordal. If the new edge is $\{v_3,v_{n-1}\}$, then we have either $\{x,v_3\}\in E(G)$ or $\{y,v_3\}\in E(G)$. However, $\{x,v_3\}\in E(G)$ implies that $\{v_2,v_3\}\in E(\widetilde{G})$, a contradiction to the fact that $(\widetilde{G})^c[v_1,\ldots,v_n]\cong C_n$. Similarly, $\{y,v_3\}\in E(G)$ together with \Cref{edge lemma} imply $\{x,v_n\}\in E(G)$, and consequently, $\{v_1,v_n\}\in E(\widetilde{G})$; again a contradiction. Therefore, $\{v_3,v_{n-1}\}$ cannot be the new edge. On the other hand, $\{x,v_n\}\notin E(G)$ since $\{v_{1},v_n\}\notin E(\widetilde{G})$. Similarly, $\{y,v_n\}\notin E(G)$ since $\{v_{n-1},v_n\}\notin E(\widetilde{G})$. Consequently, both $\{v_2,v_n\}$ and $\{v_3,v_n\}$ cannot be the new edge. Lastly, if $\{v_2,v_{n-1}\}$ is the new edge, then by \Cref{edge lemma}, $\{x,v_n\}\in E(G)$, and consequently, $\{v_1,v_n\}\in E(\widetilde{G})$, again a contradiction.
	
	\noindent
	\textbf{Case III:} $\{v_1,v_{n-2}\}\in E(\widetilde{G})\setminus E(G)$. Without loss of generality, assume that $\{x,v_{1}\},\{y,v_{n-2}\}\in E(G)$. Then, by \Cref{edge lemma}, we have $\{x,v_{2}\},\{y,v_{n-1}\}\in E(G)$. Again, as before, we have $\widetilde{G}[\{v_2,v_3,v_{n-1},v_n\}]\cong C_4$, and since $G$ is chordal, the induced subgraph $\widetilde{G}[\{v_2,v_3,v_{n-1},v_n\}]$ of $\widetilde{G}$ must have a new edge. If the new edge is $\{v_3,v_{n-1}\}$, then either $\{x,v_3\}\in E(G)$ or $\{y,v_3\}\in E(G)$. However, $\{y,v_3\}\in E(G)$ implies $\{v_2,v_3\}\in E(\widetilde{G})$, a contradiction to the fact that $(\widetilde{G})^c[v_1,\ldots,v_n]\cong C_n$. Similarly, $\{x,v_3\}\in E(G)$ together with \Cref{edge lemma} imply $\{y,v_n\}\in E(G)$, and consequently, $\{v_1,v_n\}\in E(\widetilde{G})$; again a contradiction. Therefore, $\{v_3,v_{n-1}\}$ cannot be the new edge. On the other hand, $\{x,v_n\}\notin E(G)$ since $\{v_{n-1},v_n\}\notin E(\widetilde{G})$. Similarly, $\{y,v_n\}\notin E(G)$ since $\{v_{1},v_n\}\notin E(\widetilde{G})$. Consequently, both $\{v_2,v_n\}$ and $\{v_3,v_n\}$ cannot be the new edge. Lastly, if $\{v_2,v_{n-1}\}$ is the new edge, then again by \Cref{edge lemma}, $\{y,v_n\}\in E(G)$, and consequently, $\{v_1,v_n\}\in E(\widetilde{G})$, again a~contradiction.
	
	\noindent
	\textbf{Case IV:} $\{v_2,v_{n-1}\}\in E(\widetilde{G})\setminus E(G)$. This case is similar to Case III above in terms of symmetry.

    \noindent
	Thus, we conclude that $\widetilde{G}$ is a weakly chordal graph.		
\qed

In the next four lemmas, we establish various relationships between the admissible matchings of the graphs $G$ and $\widetilde{G}$.	
	\begin{lemma}\label{aim induced lemma 1}
		Let $\{x,y\}$ be an edge of a chordal graph $G$ and let $(I(G)^{[2]}:xy)=I(\widetilde{G})$. If $M$ is an induced matching of $\widetilde{G}$, then one of the following holds:
		\begin{enumerate}[label=(\roman*)]
			\item $V(M)\cap N_G(x)=\emptyset$;
			\item $V(M)\cap N_G(y)=\emptyset$;
			\item $|\{e\in M\mid N_G(x,y)\cap e\neq\emptyset\}|=1$.
		\end{enumerate}
	\end{lemma}
	\begin{proof}
		We prove that if (i) and (ii) do not hold, then (iii) holds. Let $X=\{e\in M\mid N_G(x,y)\cap e\neq\emptyset\}$. If (i) and (ii) do not hold, then $|X|\ge 1$. Suppose if possible, $|X|\ge 2$, and $e_1,e_2\in X$. Observe that $e_1\cap e_2=\emptyset$ since $e_1,e_2\in M$. Without loss of generality, let us assume that $e_1\cap N_G(x)\neq\emptyset$, and $a\in e_1\cap N_G(x)$, for some $a\in V(G)$. If $e_2\cap N_G(y)\neq\emptyset$, then for each $b\in e_2\cap N_G(y)$, $\{a,b\}\in E(\widetilde{G})$, a contradiction to the fact that $e_1,e_2\in M$. Thus we must have $e_2\cap N_G(y)=\emptyset$. In this case, $e_2\cap N_G(x)\neq\emptyset$, and let $c\in e_2\cap N_G(x)$. Since (ii) does not hold, there exists some $e\in M\setminus\{e_2\}$ such that $e\cap N_G(y)\neq\emptyset$. Let $e\cap N_G(y)=\{d\}$ for some $d\in V(G)$. Consequently, $\{c,d\}\in E(\widetilde{G})$, again a contradiction. Thus, (iii) holds.  
	\end{proof}
	
	\begin{lemma}\label{aim induced lemma 2}
		Let $G$ be a Cohen-Macaulay chordal graph with $V(G)=W_1\sqcup W_2\sqcup \dots\sqcup W_m$ such that for each $i\in [m]$, $G[W_i]\cong K_{t_i}$, where $t_i\ge 2$. Let $x\in W_i$ such that $W_i\setminus \{x\}$ contains at least one free vertex of $G$. Let $y\in N_G[x]\setminus W_i$ and $(I(G)^{[2]}:xy)=I(\widetilde{G})$. Then $\nu_1(\widetilde{G})\le \mathrm{aim}(G,2)-1$.
	\end{lemma}
	\begin{proof}
		Let $M$ be an induced matching of $\widetilde{G}$ such that $|M|=\nu_1(\widetilde{G})$. Then, by \Cref{aim induced lemma 1}, one of the following holds.
		\begin{enumerate}[label=(\roman*)]
			\item $V(M)\cap N_G(x)=\emptyset$;
			\item $V(M)\cap N_G(y)=\emptyset$;
			\item $|\{e\in M\mid N_G(x,y)\cap e\neq\emptyset\}|=1$.
		\end{enumerate}
		It is easy to see that if (i) or (ii) holds, then $M\subseteq E(G)$, and if (iii) holds, then $|M\setminus E(G)|\le 1$. By the given hypothesis, we have $y\in W_j$ for some $j\neq i$. In this case, there exists some $v\in W_j\setminus \{y\}$ such that $v$ is a free vertex of $G$. Note that, $W_i\setminus\{x\}$ also contains a free vertex of $G$, say $u$. If (i) holds, then one can observe that $M'=M\cup\{\{x,u\}\}$ is an induced matching of $G$, and if (ii) holds, then $M''=M\cup\{\{y,v\}\}$ is an induced matching of $G$. Consequently, in both cases we have $\nu_1(\widetilde{G})\le \nu_1(G)-1\le \mathrm{aim}(G,2)-1$ by \Cref{aim property}(5). Now, suppose (iii) holds. In this case, suppose $M=\{e_1,\ldots,e_r\}$ and  $e_1\in M$ such that $e_1\cap N_G(x,y)\neq\emptyset$. Then for each $f\in M\setminus\{e_1\}$, $f,\{x,u\}$ and $f,\{y,v\}$ both forms a gap, since $N_G(u,v)\subseteq N_G(x,y)$. Thus if we take $M_1=\{\{x,u\},\{y,v\}\}$ and $M_i=\{e_i\}$ for $i>1$, then $\sqcup_{i=1}^r M_i$ forms a $2$-admissible matching of $G$. Consequently, in this case too we have $\nu_1(\widetilde{G})\le \mathrm{aim}(G,2)-1$.
	\end{proof}
	
	\begin{lemma}\label{aim induced lemma 3}
		Let $G$ be a Cohen-Macaulay chordal graph with $V(G)=W_1\sqcup W_2\sqcup \dots\sqcup W_m$ such that for each $i\in [m]$, $G[W_i]\cong K_{t_i}$, where $t_i\ge 2$. Let $x\in W_i$ be a free vertex of $G$ such that $W_i\setminus \{x\}$ contains another free vertex of $G$. Let $y\in W_i\setminus\{x\}$ and $(I(G)^{[2]}:xy)=I(\widetilde{G})$. Then $\nu_1(\widetilde{G})\le \mathrm{aim}(G,2)-1$.
	\end{lemma}
	\begin{proof}
		Let $M$ be an induced matching of $\widetilde{G}$ such that $|M|=\nu_1(\widetilde{G})$. We have the following two ~cases:
		
		\noindent
		\textbf{Case I:} $y$ is a free vertex of $G$. In this case, $N_G[x]=N_G[y]$ and $M\subseteq E(G)$. Thus by \Cref{aim induced lemma 1}, either $V(M)\cap N_G[x]=V(M)\cap N_G[y]=\emptyset$ or $|\{e\in M\mid N_G(x,y)\cap e\neq\emptyset\}|=1$. If $V(M)\cap N_G[x]=V(M)\cap N_G[y]=\emptyset$, then $M'=M\cup\{\{x,y\}\}$ is an induced matching of $G$ so that $\nu_1(\widetilde{G})\le \mathrm{aim}(G,2)-1$, by \Cref{aim property}. Now suppose $|\{e\in M\mid N_G(x,y)\cap e\neq\emptyset\}|=1$. Let $M=\{e_1,\ldots,e_r\}$ and $e_1\in M$ such that $e_1\cap N_G(x,y)\neq\emptyset$. Then for each $f\in M\setminus \{e_1\}$, the two edges $f$, $\{x,y\}$ form a gap. Hence one can take $M_1=\{e_1,\{x,y\}\}$, and $M_i=\{e_i\}$ for $i>1$ so that $\sqcup_{i=1}^r M_i$ forms a $2$-admissible matching of $G$. Consequently, $\nu_1(\widetilde{G})\le \mathrm{aim}(G,2)-1$.
		
		\noindent
		\textbf{Case II:} $y$ is not a free vertex of $G$. In this case, by assumption, there exists some $z\in W_i\setminus\{x,y\}$ such that $z$ is also a free vertex of $G$. Observe that $\{x,z\},\{y,z\}\in E(G)$, and $N_G[z]= N_G[x]\subseteq N_G[y]$. Now, by \Cref{aim induced lemma 1}, we see that one of the following holds.
		\begin{enumerate}[label=(\roman*)]
			\item $V(M)\cap N_G(x)=\emptyset$;
			\item $V(M)\cap N_G(y)=\emptyset$;
			\item $|\{e\in M\mid N_G(x,y)\cap e\neq\emptyset\}|=1$.
		\end{enumerate}
		As before, if (i) or (ii) holds, then $M\subseteq E(G)$, and if (iii) holds, then $|M\setminus E(G)|\le 1$. Now, if (i) holds, then one can observe that $N'=M\cup\{\{x,z\}\}$ is an induced matching of $G$, and if (ii) holds, then $N''=M\cup\{\{y,z\}\}$ is an induced matching of $G$. Consequently, in both cases we have $\nu_1(\widetilde{G})\le \mathrm{aim}(G,2)-1$, by \Cref{aim property}(5). Now, suppose (iii) holds. In this case, let $M=\{e_1,\ldots,e_r\}$ and $e_1\in M$ such that $e_1\cap N_G(x,y)\neq\emptyset$. Since $N_G[x]\subseteq N_G[y]$, we see that $e_1\cap N_G(y)\neq\emptyset$. Moreover, for each $f\in M\setminus\{e_1\}$, we see that $f,\{x,y\}$ form a gap. Now, if $e_1\in E(G)$, then take $M_1=\{e_1,\{x,y\}\}$, $M_i=\{e_i\}$ for $i>1$, and observe that $\sqcup_{i=1}^r M_i$ forms a $2$-admissible matching of $G$. Hence, $\nu_1(\widetilde{G})\le \mathrm{aim}(G,2)-1$. If $e_1\notin E(G)$, then $e_1=\{p,q\}$, where $p\in W_i\setminus\{x,y\}$, $q\in W_l$ for some $l\neq i$, and $\{y,q\}\in E(G)$. Thus one can take $N_1=\{\{x,p\},\{y,q\}\}, N_i=\{e_i\}$ for $i>1$, and observe that $\sqcup_{i=1}^r N_i$ forms a $2$-admissible matching of $G$ so that $\nu_1(\widetilde{G})\le \mathrm{aim}(G,2)-1$. This completes the proof.
	\end{proof}
	
	\begin{lemma}\label{aim induced lemma 4}
		Let $G$ be any graph and $x,y\in V(G)$ such that $y\in N_G(x)$ and $N_G[y]\subseteq N_G[x]$. If $H$ denotes the graph $G\setminus N_G[x]$, then for each $k\ge 1$, we have $\mathrm{aim}(H,k)\le\mathrm{aim}(G,k)-1$.
	\end{lemma}
	\begin{proof}
		Let $M=\sqcup_{i=1}^r M_i$ be a $k$-admissible matching of $H$ such that $|M|=\mathrm{aim}(H,k)$. Then for each $e\in M$, $e\cap N_G(x,y)=\emptyset$ since $N_G[y]\subseteq N_G[x]$. Hence $M'=\sqcup_{i=1}^{r+1} M_i'$ is a $k$-admissible matching of $G$, where $M_i'=M_i$ for each $i\in[r]$, and $M_{r+1}'=\{\{x,y\}\}$. Thus $\mathrm{aim}(H,k)\le\mathrm{aim}(G,k)-1$. 
	\end{proof}

    We are now in a position to prove the main result of this section.

\begin{theorem}\label{CM chordal regularity}
Let $G$ be a Cohen-Macaulay chordal graph with the matching number at least $2$. Then 
\[
\mathrm{reg}(I(G)^{[2]})=\mathrm{aim}(G,2)+2.
\]
\end{theorem}
\begin{proof}
By \Cref{thm: ordinary lower bound d uniform}, we have $\mathrm{reg}(I(G)^{[2]})\ge\mathrm{aim}(G,2)+2$. To show $\mathrm{reg}(I(G)^{[2]})\le\mathrm{aim}(G,2)+2$, we use induction on $|V(G)|$. Since $G$ has matching number at least $2$, we have $|V(G)|\ge 4$. Now, if $|V(G)|=4$, then $I(G)^{[2]}$ is a principal ideal, and thus it is easy to see that $\mathrm{reg}(I(G)^{[2]})=\mathrm{aim}(G,2)+2$. Now suppose $|V(G)|\ge 5$. Let $V(G)=W_1\sqcup W_2\sqcup \dots\sqcup W_m$ such that for each $i\in [m]$, $G[W_i]\cong K_{t_i}$, where $t_i\ge 2$. Then we have the following two cases:

\noindent
\textbf{Case I:} $G$ is not a disjoint union of complete graphs. Then there exists some $i\in [m]$ and $x\in W_i$ such that $N_G(x)\setminus W_i\neq\emptyset$. Without loss of generality, let $i=1$. Moreover, let $W_1=\{x_1,\ldots,x_r\}$ for some $r\ge 2$ such that $x_r$ is a free vertex of $G$, and $N_G(x_1)\setminus W_1\neq\emptyset$. Suppose $N_G(x_1)\setminus W_1=\{y_1,\ldots,y_s\}$.

\noindent
\textbf{Claim 1:} For each $l\in[s-1]\cup\{0\}$, 
\[
\mathrm{reg}((I(G)^{[2]}+\langle x_1y_1,\ldots,x_1y_l\rangle):x_1y_{l+1})\le\mathrm{aim}(G,2),
\]
where $I(G)^{[2]}+\langle x_1y_1,\ldots,x_1y_l\rangle=I(G)^{[2]}$ in case $l=0$.

\noindent
\textit{Proof of Claim 1}: Let $G_l=G\setminus\{y_1,\ldots,y_{l}\}$. Then, it is easy to see that $G_l$ is a Cohen-Macaulay chordal graph. Now,
\begin{align*}
	\reg((I(G)^{[2]}+\l x_1y_1,\ldots,x_1y_l\r):x_1y_{l+1})=\reg(I(\widetilde{G_l}))
	&= \nu_1(\widetilde{G_l})+1\\
    &\le \aim(G_l,2)\\
	&\le \aim(G,2),
\end{align*}
where $\widetilde{G_l}$ is a weakly chordal graph (by \Cref{weakly chordal}), the first equality follows from \cite[Lemma 2.9]{DRS20242}, the second equality follows from \Cref{aim property}, the first inequality follows from \Cref{aim induced lemma 2}, and the second inequality follows from \Cref{admissible result2}. This completes the proof of Claim~1.

\noindent
\textbf{Claim 2:} For each $p\in[r-1]$, 
\[
\mathrm{reg}((I(G)^{[2]}+\langle x_1y_1,\ldots,x_1y_s,x_1x_2,\ldots,x_1x_p\rangle):x_1x_{p+1})\le\mathrm{aim}(G,2),
\]
where $I(G)^{[2]}+\langle x_1y_1,\ldots,x_1y_s,x_1x_2,\ldots,x_1x_p\rangle=I(G)^{[2]}+\langle x_1y_1,\ldots,x_1y_s\rangle$ in case $p=1$.

\noindent
\textit{Proof of Claim 2}: Let $H_p=G\setminus\{y_1,\ldots,y_{s},x_2,\ldots,x_p\}$. Then again, it is easy to see that $H_p$ is a Cohen-Macaulay chordal graph. Moreover, both $x_{1}$ and $x_{r}$ are free vertices of $H_p$. Now,
\begin{align*}
	\mathrm{reg}((I(G)^{[2]}+\langle x_1y_1,\ldots,x_1y_s,x_1x_2,\ldots,x_1x_p\rangle):x_1x_{p+1})=\reg(I(\widetilde{H_p}))
	&= \nu_1(\widetilde{H_p})+1\\
    &\le \aim(H_p,2)\\
	&\le \aim(G,2),
\end{align*}
where $\widetilde{H_p}$ is a weakly chordal graph (by \Cref{weakly chordal}), the equality follows \cite[Lemma 2.9]{DRS20242}, the second equality follows from \Cref{aim property}, the first inequality follows from \Cref{aim induced lemma 3},  and the second inequality follows from \Cref{admissible result2}. This completes the proof of Claim 2.

Let $\Gamma=G\setminus N_G[x_1]$. Then $\Gamma$ is a Cohen-Macaulay chordal graph. Moreover,
\begin{equation}\label{eq31}
	\begin{split}
		\reg((I(G)^{[2]}+\langle x_1y_1,\ldots,x_1y_s,x_1x_2,\ldots,x_1x_r\rangle):x_1)&=\reg(I(\Gamma)^{[2]}+\l y_1,\ldots,y_s,x_2,\ldots,x_r\r)\\
		&\le \mathrm{aim}(\Gamma,2)+2\\
		&\le \aim(G,2)+1,
	\end{split}
\end{equation}
where the equality again follows from \cite[Lemma 2.9]{DRS20242}, the first inequality follows from the induction hypothesis, and the second inequality follows from \Cref{aim induced lemma 4}.

Next, we have
\begin{equation}\label{eq32}
	\begin{split}
		\reg(I(G)^{[2]}+\l x_1y_1,\ldots,x_1y_s,x_1x_2,\ldots,x_1x_r,x_1\r)&=\reg(I(G\setminus x_1)^{[2]})\\
		&\le\aim(G\setminus x_1,2)+2\\
		&\le \aim(G,2)+2,
	\end{split}
\end{equation}
where the first and second inequalities follow from the induction hypothesis and \Cref{admissible result2},~respectively.

Now, using \Cref{regularity lemma} along with inequalities~(\ref{eq31}) and (\ref{eq32}), we get 
\[\reg(I(G)^{[2]}+\l x_1y_1,\ldots,x_1y_s,x_1x_2,\ldots,x_1x_r\r)\le\aim(G,2)+2.\] 
Applying \Cref{regularity lemma} with the above inequality and the inequality proved in Claim 2, we get \[\reg(I(G)^{[2]}+\l  x_1y_1,\ldots,x_1y_s,x_1x_2,\ldots,x_1x_{r-1}\r)\le\aim(G,2)+2.\] 
By a repeated use of Claim 2 and \Cref{regularity lemma}, we get $\reg(I(G)^{[2]}+\l x_1y_1,\ldots,x_1y_s\r)\le\aim(G,2)+2$. Now, from Claim 1 we have $\reg((I(G)^{[2]}+\l x_1y_1,\ldots,x_1y_{s-1}\r):x_1y_{s})\le\aim(G,2)$. Thus using \Cref{regularity lemma} again we obtain $\reg((I(G)^{[2]}+\l x_1y_1,\ldots,x_1y_{s-1}\r)\le\aim(G,2)+2$. By a repeated use of Claim 1 and \Cref{regularity lemma} this time, we finally get $\reg(I(G)^{[2]}\le\aim(G,2)+2$, as~desired.

\noindent
\textbf{Case II:} $G$ is a disjoint union of complete graphs. In this case, let $W_1=\{x_1,\ldots,x_r\}$ for some $r\ge 2$. Then, proceeding as in Case I above, we obtain the following:
\begin{enumerate}[label=(\roman*)]
	\item For each $p\in[r-1]$, $\mathrm{reg}((I(G)^{[2]}+\langle x_1x_2,\ldots,x_1x_p\rangle):x_1x_{p+1})\le\mathrm{aim}(G,2)$, where $I(G)^{[2]}+\langle x_1x_2,\ldots,x_1x_p\rangle=I(G)^{[2]}$ in case $p=1$;
	
	\item $ \mathrm{reg}((I(G)^{[2]}+\langle x_1x_2,\ldots,x_1x_r\rangle):x_1)\le\mathrm{aim}(G,2)+1$;
	
	\item $ \mathrm{reg}((I(G)^{[2]}+\langle x_1x_2,\ldots,x_1x_r,x_1\rangle)\le\mathrm{aim}(G,2)+2$.
\end{enumerate}
Now, using \Cref{regularity lemma} along with (ii) and (iii) above, we have $\mathrm{reg}(I(G)^{[2]}+\langle x_1x_2,\ldots,x_1x_r\rangle)\le \mathrm{aim}(G,2)+2$. Thus, as before, by a repeated use of \Cref{regularity lemma} and (i) above we obtain $\reg(I(G)^{[2]})\le\aim(G,2)+2$. This completes the proof of the theorem.
\end{proof}

\section{Concluding remarks and some open problems}\label{section 6}

Let $I$ be a (not necessarily squarefree) monomial ideal and $k$ a positive integer. Erey and Ficarra \cite{ErFi1} recently defined the \emph{$k$-mathching power} of $I$ to be the ideals generated by $m_1m_2\cdots m_k$ where $m_1,m_2,\dots, m_k\in \G(I)$ have pairwise-disjoint supports. They denoted the $k$-th matching power of $I$ by $I^{[k]}$ as it is clear from the definition that if $I$ is square-free, then the $k$-th matching power of $I$ is exactly its $k$-th square-free power. Recall from \cite[cf. Exercise 6.4.33]{RHV} that $\reg I^{\P}=\reg I$, where $I^{\P}$, a square-free monomial ideal, denotes the \emph{polarization} of $I$. By \cite[Proposition~1.4]{ErFi1}, for any monomial ideal $I$ and integer $k$, we have $\left(I^{[k]}\right)^{\P}=\left(I^\P \right)^{[k]}$. Thus, one can use \Cref{thm: ordinary lower bound} to give a lower bound for the regularity of matching powers of monomial ideals. We leave the details to interested readers.

For several classes of graphs, the regularity of their edge ideals can be expressed in terms of the induced matching number of the corresponding graphs \cite{BKH1}. For instance, $\reg(I(G))=\nu_1(G)+1$ if $G$ belongs to one of the following classes:
\begin{enumerate}[label=(\roman*)]
    \item (weakly) chordal;
    \item very well-covered;
    \item sequentially Cohen-Macaulay bipartite;
    \item $C_5$-free vertex decomposable;
    \item certain unicyclic graphs.    
\end{enumerate}

\noindent For any graph $G$, it is well-known that $\reg(I(G)^k)\geq \nu_1(G)+2k-1$. Moreover, it is expected that this lower bound will be attained for all the above-mentioned classes of graphs.
\par 

In the context of square-free powers, it was proved in \Cref{thm: ordinary lower bound d uniform} that $\reg(I(G)^{[k]})\geq \aim(G,k)+k$ for all $1\leq k\leq \nu(G)$. Note that $\aim(G,1)=\nu_{1}(G)$. Thus, $\aim(G,1)+1$ is indeed $\reg(I(G)^{[1]})$ when $G$ is one of the above graphs. Hence, analogous to ordinary powers, we ask the following.

\begin{question}\label{ques}
    Let $G$ be a simple graph such that $\reg(I(G))=\nu_1(G)+1$. Then is it true that $\reg(I(G)^{[k]})=\aim(G,k)+k$ for any $1\leq k\leq \nu(G)$?
\end{question}

We have an affirmative answer to the above question for block graphs, as seen in \Cref{block ordinary upper bound}. Also, for a Cohen-Macaulay chordal graph $G$, we have shown in \Cref{CM chordal regularity} that $\reg(I(G)^{[2]})=\aim(G,2)+2$. Moreover, we have checked all chordal graphs up to $8$ vertices and observed that \Cref{ques} has a positive answer. Based on all these, we conjecture the following.

 \begin{conjecture}\label{conj: chordal reg}
      Let $G$ be a chordal graph. Then, for all $1\leq k\leq \nu(G)$, we have
        \[ \reg(I(G)^{[k]})=\aim(G,k)+k. \]
 \end{conjecture}

\section*{Acknowledgements}

Chau, Das, and Roy are supported by Postdoctoral Fellowships at the Chennai Mathematical Institute. Saha would like to thank the National Board for Higher Mathematics (India) for the financial support through the NBHM Postdoctoral Fellowship. All the authors are partially supported by a grant from the Infosys Foundation.  



\bibliographystyle{abbrv}
\bibliography{ref}

\begin{thebibliography}{10}

\bibitem{BKH1}
A.~Banerjee, S.~K. Beyarslan, and H.~T. H\`a.
\newblock Regularity of edge ideals and their powers.
\newblock In {\em Advances in algebra}, volume 277 of {\em Springer Proc. Math. Stat.}, pages 17--52. Springer, Cham, 2019.

\bibitem{BCDMS2022}
A.~Banerjee, B.~Chakraborty, K.~K. Das, M.~Mandal, and S.~Selvaraja.
\newblock Regularity of powers of squarefree monomial ideals.
\newblock {\em J. Pure Appl. Algebra}, 226(2):Paper No. 106807, 12, 2022.

\bibitem{Bigdeli2018(1)}
M.~Bigdeli, J.~Herzog, and R.~Zaare-Nahandi.
\newblock On the index of powers of edge ideals.
\newblock {\em Comm. Algebra}, 46(3):1080--1095, 2018.

\bibitem{BM76}
J.~A. Bondy and U.~S.~R. Murty.
\newblock {\em Graph theory with applications}.
\newblock American Elsevier Publishing Co., Inc., New York, 1976.

\bibitem{CDRS-symbolic}
T.~Chau, K.~K. Das, A.~Roy, and K.~Saha.
\newblock On the {C}astelnuovo-{M}umford regularity of symbolic square-free powers, work in preparation.

\bibitem{CFL1}
M.~Crupi, A.~Ficarra, and E.~Lax.
\newblock Matchings, square-free powers and betti splittings.
\newblock {\em arXiv:$2304.00255$}, 2023.

\bibitem{DHS}
H.~Dao, C.~Huneke, and J.~Schweig.
\newblock Bounds on the regularity and projective dimension of ideals associated to graphs.
\newblock {\em J. Algebraic Combin.}, 38(1):37--55, 2013.

\bibitem{DRS20242}
K.~K. Das, A.~Roy, and K.~Saha.
\newblock Square-free powers of {C}ohen-{M}acaulay forests, cycles, and whiskered cycles.
\newblock {\em arXiv:$2409.06021$}, 2024.

\bibitem{CMSimplicialForests}
K.~K. Das, A.~Roy, and K.~Saha.
\newblock Square-free powers of cohen-macaulay simplicial forests.
\newblock {\em arXiv:$2502.18396$}, 2025.

\bibitem{Edmonds}
J.~Edmonds.
\newblock An introduction to matching.
\newblock Available at \url{https://web.eecs.umich.edu/~pettie/matching/Edmonds-notes.pdf}, 1967.

\bibitem{EliKer}
S.~Eliahou and M.~Kervaire.
\newblock Minimal resolutions of some monomial ideals.
\newblock {\em Journal of Algebra}, 129(1):1--25, 1990.

\bibitem{ErFi1}
N.~Erey and A.~Ficarra.
\newblock Matching powers of monomial ideals and edge ideals of weighted oriented graphs.
\newblock {\em arXiv:$2309.13771$}, 2024.

\bibitem{EHHS}
N.~Erey, J.~Herzog, T.~Hibi, and S.~Saeedi~Madani.
\newblock Matchings and squarefree powers of edge ideals.
\newblock {\em J. Combin. Theory Ser. A}, 188:Paper No. 105585, 24, 2022.

\bibitem{EHHM12}
N.~Erey, J.~Herzog, T.~Hibi, and S.~Saeedi~Madani.
\newblock The normalized depth function of squarefree powers.
\newblock {\em Collect. Math.}, 75(2):409--423, 2024.

\bibitem{ErHi1}
N.~Erey and T.~Hibi.
\newblock Squarefree powers of edge ideals of forests.
\newblock {\em Electron. J. Combin.}, 28(2):Paper No. 2.32, 16, 2021.

\bibitem{Fatabbi2001}
G.~Fatabbi.
\newblock On the resolution of ideals of fat points.
\newblock {\em J. Algebra}, 242(1):92--108, 2001.

\bibitem{FiHeHi}
A.~Ficarra, J.~Herzog, and T.~Hibi.
\newblock Behaviour of the normalized depth function.
\newblock {\em Electron. J. Combin.}, 30(2):Paper No. 2.31, 16, 2023.

\bibitem{ficarraCM2024}
A.~Ficarra and S.~Moradi.
\newblock Monomial ideals whose all matching powers are {C}ohen-{M}acaulay, 2024.

\bibitem{FranciscoHaVanTuyl2009}
C.~A. Francisco, H.~T. H\`a, and A.~Van~Tuyl.
\newblock Splittings of monomial ideals.
\newblock {\em Proc. Amer. Math. Soc.}, 137(10):3271--3282, 2009.

\bibitem{HaVanTuyl2008}
H.~T. H\`a and A.~Van~Tuyl.
\newblock Monomial ideals, edge ideals of hypergraphs, and their graded {B}etti numbers.
\newblock {\em J. Algebraic Combin.}, 27(2):215--245, 2008.

\bibitem{CMchordal}
J.~Herzog, T.~Hibi, and X.~Zheng.
\newblock Cohen-{M}acaulay chordal graphs.
\newblock {\em J. Combin. Theory Ser. A}, 113(5):911--916, 2006.

\bibitem{KaNaQu}
E.~Kamberi, F.~Navarra, and A.~A. Qureshi.
\newblock On squarefree powers of simplicial trees.
\newblock {\em arXiv:$2406.13670$}, 2024.

\bibitem{MoreyVillarreal}
S.~Morey and R.~H. Villarreal.
\newblock Edge ideals: algebraic and combinatorial properties.
\newblock In {\em Progress in commutative algebra 1}, pages 85--126. de Gruyter, Berlin, 2012.

\bibitem{SaS1}
S.~A. Seyed~Fakhari.
\newblock On the {C}astelnuovo-{M}umford regularity of squarefree powers of edge ideals.
\newblock {\em J. Pure Appl. Algebra}, 228(3):Paper No. 107488, 12, 2024.

\bibitem{Fakhari2024}
S.~A. {Seyed Fakhari}.
\newblock {On the Castelnuovo–Mumford regularity of squarefree powers of edge ideals}.
\newblock {\em Journal of Pure and Applied Algebra}, 228(3):107488, 2024.

\bibitem{FakhariSymbSq-free}
S.~A. Seyed~Fakhari.
\newblock On the regularity of squarefree part of symbolic powers of edge ideals.
\newblock {\em J. Algebra}, 665:103--130, 2025.

\bibitem{RHV}
R.~H. Villarreal.
\newblock {\em Monomial algebras}.
\newblock Monographs and Research Notes in Mathematics. CRC Press, Boca Raton, FL, second edition, 2015.

\end{thebibliography}
\end{document}